%% file: before_journal.tex
\documentclass[11pt]{article}

\usepackage[margin=1in]{geometry}
\usepackage{amsmath,amssymb,amsthm,mathtools}
\usepackage{enumitem}
\usepackage{xcolor}
\usepackage{microtype}
\usepackage{booktabs}
\usepackage[hidelinks]{hyperref}
\usepackage[T1]{fontenc}
\usepackage{lmodern}
\usepackage{algorithm}
\usepackage{algpseudocode}
\usepackage{float}
\usepackage{comment}
\usepackage[title]{appendix}
\usepackage{thmtools}
\usepackage{thm-restate}
\usepackage{tikz}
\usetikzlibrary{arrows.meta,decorations.pathreplacing,positioning}
\usepackage{adjustbox}
\allowdisplaybreaks

\newtheorem{theorem}{Theorem}[section]
\newtheorem{lemma}[theorem]{Lemma}
\newtheorem{proposition}[theorem]{Proposition}
\newtheorem{corollary}[theorem]{Corollary}
\newtheorem{claim}[theorem]{Claim}
\theoremstyle{definition}
\newtheorem{definition}[theorem]{Definition}
\theoremstyle{remark}
\newtheorem{remark}[theorem]{Remark}
\newtheorem{example}[theorem]{Example}

\newcommand{\R}{\mathbb{R}}

\newcommand{\widthf}{\operatorname{wd}}
\newcommand{\conv}{\operatorname{conv}}
\newcommand{\proj}{\operatorname{proj}}
\newcommand{\rec}{\operatorname{rec}}
\newcommand{\aff}{\operatorname{aff}}

\newcommand{\eps}{\varepsilon}
\newcommand{\ind}{\chi}
\newcommand{\cL}{c_{\mathcal L}}

\newcommand{\gap}{\operatorname{gap}}
\newcommand{\oo}[1]{{\textcolor{purple}{#1}}}
\newcommand{\ooc}[1]{{\textcolor{purple}{[OG: \oo{#1}]}}}

\newcommand{\ignore}[1]{}

\title{Boxing Problem}

\author{
    April Niu\thanks{
    \raggedright
        School of Industrial and Systems Engineering, Georgia Institute of Technology. \texttt{yniu74@gatech.edu}.
    }
    \and
    Santanu Dey\thanks{
    \raggedright
        School of Industrial and Systems Engineering, Georgia Institute of Technology.
        \texttt{santanu.dey@isye.gatech.edu}.
    }
    \and
    Oktay G{\"u}nl{\"u}k\thanks{
    \raggedright
    School of Industrial and Systems Engineering, Georgia Institute of Technology.
   \texttt{oktay.gunluk@isye.gatech.edu}
    }
    \and
    Jisun Lee\thanks{
    \raggedright
    Walker Department of Mechanical Engineering, University of Texas at Austin.
    \texttt{jisun.lee@austin.utexas.edu}
    }
}

\date{\today}
\begin{document}

\maketitle
\begin{abstract}
\input{abstract}

\end{abstract}

\input{intro}

\input{mixing_v2}
\input{properties}

\input{1Dboxing}

\input{higher_dim}

\input{discussion}

\bibliographystyle{acm}
\bibliography{ref}
\newpage
\begin{appendices}
\input{appendix}
\end{appendices}

\end{document}

%% file: abstract.tex
We study the \textit{Boxing Problem}, which seeks a subset of points in 
\(\mathbb{R}^d\) maximizing the total reward of a selection of points minus the width of its axis-aligned bounding box. We also investigate the associated \textit{Boxing Set}, defined by the epigraph of the width function over binary selection variables. Although the width function is not submodular in general, we restore submodularity by fixing an anchor point. After anchoring, we establish a connection to the intersection of mixing sets with a linking constraint, which yields a polynomial-time optimization algorithm and a polynomial-size extended formulation. We then study the polyhedral structure of the Boxing Set and show that its facets admit laminar families of tight sets and can be obtained by sequential lifting. For the one-dimensional case, we provide a complete description of facet-defining inequalities together with a strongly polynomial-time separation algorithm.

%% file: intro.tex
\section{Introduction}\label{sec:intro}

Let $p^1,\ldots,p^n\in\R^d$ be an indexed collection of points, with
	$p^i=(p^i_1,\ldots,p^i_d)$ for each $i\in[n]$. For each nonempty subset
	$S\subseteq[n]$ and coordinate $t\in[d]$, define
\[
 u_t(S):=\max_{i\in S}p_t^i,\qquad
 \ell_t(S):=\min_{i\in S}p_t^i,\qquad
 \widthf_t(S):=u_t(S)-\ell_t(S).
\]
Set $\widthf_t(\varnothing):=0$ for all $t\in[d]$, and and define the total
	width of any subset $S\subseteq[n]$ as
$\widthf(S):=\sum_{t=1}^d\widthf_t(S)$.
In other words,  $\widthf(S)$ is the sum of the side lengths of the smallest axis-aligned box containing the points indexed
by $S$. Given nonnegative rewards $c_i$ for $i\in[n]$, let
$c(S):=\sum_{i\in S}c_i$ denote the total reward	of the points indexed by $S$. The \emph{Boxing Problem} is

\begin{definition}[Boxing Problem]\label{def:boxing-problem}
Given $p^1,\ldots,p^n\in\R^d$ and $c\in\R_+^n$, compute
\[
\operatorname{BP}(c):=\max_{S\subseteq[n]}\{c(S)-\widthf(S)\}.
\]
\end{definition}

For $z\in\{0,1\}^n$, let $S(z):=\{i\in[n]:z_i=1\}$. For
$S\subseteq[n]$, let $\ind^S$ denote its incidence vector. We study the
convex hull of the associated feasible set.

\begin{definition}[Boxing set]
The \emph{Boxing Set} is
\begin{equation*}\label{eq:def-B}
B:=\{(w,z)\in\R\times\{0,1\}^n:w\ge\widthf(S(z))\}.
\end{equation*}
\end{definition}

For $J\subseteq[n]$, we also define the restricted boxing set
\begin{equation*}\label{eq:restricted-boxing-set}
B[J]:=B\cap\{(w,z)\in\R\times\{0,1\}^n:
z_i=0\text{ for every }i\notin J\}.
\end{equation*}

The Boxing Problem arises naturally in hyper-rectangular
clustering, where points are partitioned into clusters
to minimize the sum of the widths of their axis-aligned
bounding boxes~\cite{DelleDonneMarenco2025,Marenco2023}.
Delle Donne and Marenco~\cite{DelleDonneMarenco2025} develop a branch-and-price
algorithm for this problem.

To illustrate the connection, consider partitioning
the points into exactly $K$ nonempty clusters. Let
$\mathcal C=2^{[n]}\setminus\{\emptyset\}$ and introduce
a binary variable $y_S$ for each potential cluster
$S\in\mathcal C$. A set-partitioning formulation is
{\begin{equation*}
\min\ \sum_{S\in\mathcal C}\widthf(S)y_S \quad \text{s.t.}\quad \sum_{S\in\mathcal C:~i\in S}y_S=1~ \forall i\in[n],\quad \sum_{S\in\mathcal C}y_S=K,\quad y_S\in\{0,1\}~\  \forall S\in\mathcal C.
\end{equation*}}
\ignore{\begin{equation}
\begin{aligned}
\min\quad
& \sum_{S\in\mathcal C}\widthf(S)y_S\\
\text{s.t.}\quad
& \sum_{\substack{S\in\mathcal C\\i\in S}}y_S=1,
&& i\in[n],\\[.15cm]
& \sum_{S\in\mathcal C}y_S=K,\\
& y_S\in\{0,1\},
&& S\in\mathcal C.
\end{aligned}
\end{equation}}
Since the number of possible clusters is exponential,
column generation can be used to solve the linear
relaxation. Let $c_i$ be the dual variable associated
with the assignment constraint of point $i$, and let
$\gamma$ be the dual variable for the cluster-count
constraint. Then, finding a column of minimum reduced cost amounts to solving \(\min_{\emptyset\neq S\subseteq[n]}{wd(S)-c(S)}\), up to the additive constant \(-\gamma\). Since the dual variables \(c_i\) are unrestricted in sign, this is a signed-reward variant of the Boxing Problem restricted to nonempty subsets. 

The width function is not submodular on all of $2^{[n]}$, since the empty
intersection of two nonempty sets can violate the submodular inequality.
However, it satisfies that inequality whenever the two sets intersect.
We use this structure, together with a connection to mixing sets, to analyze the Boxing Set.

\subsection{Related work}

Submodular set functions have a well-developed optimization and polyhedral theory. Edmonds~\cite{Edmonds1970} established the greedy algorithm for optimizing linear functions over submodular polyhedra, while the Lov\'asz extension gives a convex representation of a submodular function and describes the convex hull of its binary epigraph through extended polymatroid
inequalities~\cite{Lovasz1983}. Atamt\"urk and
Narayanan~\cite{AtamturkNarayanan2022} develop a polarity-based
approach to set-function epigraphs. Their polar relaxation,
based on homogeneous inequalities, is exact for submodular
functions but generally provides only an outer approximation
for non-submodular functions. Related polyhedral
convexification results have also been developed for broader
classes of submodular functions
\cite{YuKucukyavuz2021,YuKucukyavuz2023,YuKucukyavuz2025}.
In contrast, the box-width function is not globally
submodular: it satisfies the submodular inequality for
intersecting sets, but the inequality may fail for disjoint
sets. Consequently, the standard submodular epigraph
description does not directly characterize the Boxing Set.
Our approach  considers the normalized polar
coefficient polyhedron, which explicitly includes the
right-hand-side variable needed to characterize nontrivial
facets.

The algorithmic theory of submodular functions also provides
important tools for our analysis. Edmonds' greedy algorithm
and its extensions to intersecting-submodular polyhedra are
closely related to Dilworth truncation and the optimization
of partition functions~\cite{Edmonds1970,Fujishige2005}.
Nagano, Kawahara, and Iwata~\cite{NaganoKawaharaIwata2010}
use these connections to study parametric partition
optimization. Strongly polynomial algorithms for
submodular-function minimization are also available
\cite{Orlin2009}. We exploit these techniques and provide a modified greedy procedure
based on submodular-function minimization to compute the facet right-hand sides.

Although submodularity is restricted, the Boxing Set has a rich combinatorial structure that allows \textit{uncrossing arguments}.
Uncrossing is widely used in polyhedral combinatorics and
approximation algorithms to establish laminar families of
tight constraints. For example, Jain~\cite{Jain2001} shows
that the tight cut constraints of a basic LP solution in
survivable network design admit a laminar basis, which is
central to his iterative-rounding algorithm. Related
techniques appear in primal-dual network-design algorithms
\cite{GoemansWilliamson1995,WilliamsonEtAl1995} and are
developed systematically by Lau, Ravi, and
Singh~\cite{LauRaviSingh2011}. Our use of uncrossing is
different: rather than analyzing the tight constraints of a
basic LP solution, we uncross the tight constraints of the
normalized polar coefficient polyhedron associated with a
facet of the Boxing Set. {This is possible because the width function is intersecting
submodular, even though it
is not  submodular.} The resulting laminar family
uniquely determines the facet coefficients. By further
analyzing its structure, we prove that every nontrivial
facet can be obtained through sequential lifting. Thus,
our approach extends the polarity-based viewpoint of
Atamt\"urk and Narayanan~\cite{AtamturkNarayanan2022}
from homogeneous inequalities to a structural analysis of
nontrivial facets with positive right-hand sides.

A separate line of related work concerns mixing sets and
their intersections. G\"unl\"uk and Pochet~\cite{GP01}
develop mixing inequalities, while Atamt\"urk, Nemhauser,
and Savelsbergh~\cite{an00} study systems with multiple
continuous variables sharing common binary variables.
Conforti, Wolsey, and Zambelli
\cite{ConfortiWolseyZambelli2010} and
K\"u\c{c}\"ukyavuz~\cite{Kucukyavuz2012} investigate
related mixing-set formulations and additional constraints.
Most closely related to our work, K\i l\i n\c{c}-Karzan,
K\"u\c{c}\"ukyavuz, and Lee~\cite{kkl22} establish a
connection between mixing inequalities and submodularity,
and study intersections of mixing sets with common binary
variables under an additional linking constraint. Their approach
exploits the submodular structure underlying mixing sets.
However, it does not directly characterize the Boxing Set,
because the original box-width function is not submodular.
We overcome this difficulty through \textit{anchoring}:
fixing one selected point expresses the width as a sum of
one-sided maxima, each of which is submodular. Thus, every anchored Boxing Set can be represented by an intersection
of mixing sets with an epigraph linking constraint.
Combining these anchored sets by disjunctive programming
then yields a compact extended formulation for the original
Boxing Set. Beyond this mixing-set representation, our
polyhedral analysis directly addresses the original
non-submodular width function, leading to the laminar
facet structure and the complete one-dimensional
convex-hull description.

In recent independent work, Wu, He, and Tawarmalani~\cite{WuHeTawarmalani2026} develop a
convexification framework for disjunctive submodular
functions with any finite number of disjuncts.
They obtain compact LP separation formulations when
the component epigraph convex hulls admit compact
formulations, and strongly polynomial separation for
the minimum of two submodular functions on a common
lattice. Our anchoring construction fits their
pointwise-minimum framework, while their greedy
construction and Dilworth truncation connect to our
fixed-right-hand-side lifting and partition analysis.
We exploit the geometry of the Boxing Problem to
obtain sparse maximum-weight-closure formulations,
a compact extended formulation, and a minimum-cost-flow
separation algorithm in one dimension. Our polyhedral
analysis further establishes laminar facet structure,
a complete one-dimensional facet description, and
a characterization of facet right-hand-side parameters
through partitions of tight sets.

\subsection{Contributions}
 We establish new optimization and polyhedral results for the Boxing Set, whose underlying width function is a monotone set function that is not globally submodular.

\begin{itemize}
\item \textbf{Optimization and extended formulation.}
We reduce optimization over the intersection of multiple mixing sets to maximum-weight graph closure.
This connection to mixing sets yields a maximum-weight-closure algorithm for the Boxing Problem and a polynomial-size extended formulation for its convex hull.

\item \textbf{General facet structure.} We develop a dimension-independent structural analysis of the facets of the Boxing Set. We show that the tight sets of every nontrivial facet admit a laminar basis. We then prove that every such facet can be obtained through sequential lifting in an ordering induced by the corresponding laminar forest.

\item \textbf{Complete description in one dimension.}
We characterize all possible  right-hand sides for facet defining inequalities and identify exactly which lifting sequences would yield facet defining inequalities that give the convex hull. We also give a strongly polynomial
minimum-cost-flow separation algorithm.

\item \textbf{Higher-dimensional extensions.}
We characterize facet right-hand sides through partitions of tight sets,
give polynomial-time algorithms for computing the associated partition
parameter for a fixed subset and the smallest facet right-hand side, and
identify additional families of facets.
\end{itemize}

The remainder of the paper is organized as follows.
Section~\ref{sec:multi-mixing} revisits extended formulations and
develops optimization algorithms for the intersection of mixing sets over common binary variables. Section~\ref{sec:ef}
applies these results to the Boxing Set. Section~\ref{sec:basic}
studies general polyhedral properties, including laminarity and sequential
lifting. Section~\ref{sec:onedim} gives the complete one-dimensional
facet description and separation algorithm, and
Section~\ref{sec:higher-dim} investigates the higher-dimensional case.

%% file: mixing_v2.tex
\section{Multi-Mixing Sets}\label{sec:multi-mixing}

We begin by providing an algorithmic reformulation of
known mixing-set structure~\cite{GP01} that is particularly useful
for the Boxing Problem. For multiple mixing sets
sharing common binary variables and having independent
continuous variables, existing convex-hull results
yield a compact extended formulation. We then exploit
the ordered structure of the mixing coefficients to
reduce linear optimization over their intersection
to a maximum-weight graph closure problem, which can
be solved by a minimum $s$--$t$ cut.

This approach is closely related to the graph-based
formulations of Conforti, Wolsey, and
Zambelli~\cite{ConfortiWolseyZambelli2010}, who study
mixed-integer covers on bipartite graphs and their
equivalence to network-dual formulations. While their
work focuses on deriving convex-hull descriptions
through projections of extended formulations, our
focus is on constructing a sparse closure graph that
directly solves the multi-mixing optimization problem.
This representation will be used in Section~\ref{sec:ef} to
solve the Boxing Problem after fixing an anchor point.

Let $k\ge1$. For each $r\in[k]$, let $q^r=(q_r^i)_{i\in[n]}\in\mathbb{R}_+^n$. Consider $k$ mixing sets \cite{GP01} sharing the same binary
vector $z$:
\[
\operatorname{MIX}_r({q})
:=
\conv\left\{
(u,z)\in\mathbb{R}^k\times\{0,1\}^n:
u_r\ge q^i_r z_i
\quad \forall i\in[n]
\right\}.
\]

We call the set
\begin{equation}\label{eq:multi-mixing-set}
\operatorname{MMIX}(q)
:=
\left\{(u,z)\in\R^k\times\{0,1\}^n:
 u_r\ge q_r^i z_i\quad\forall i\in[n],\ r\in[k]\right\}
\end{equation}
the \emph{multi-mixing set}. It is the intersection of $k$ mixing
constraints with common binary variables and independent continuous
variables. The convex hull of such intersections was studied by
Atamt\"urk, Nemhauser, and Savelsbergh~\cite{an00}, and later by
K\i l\i n\c c-Karzan, K\"u\c c\"ukyavuz, and Lee~\cite{kkl22}.
In particular, for this multi-mixing set without additional linking
constraints, their results give
\begin{equation}\label{eq:multi-mixing-intersection}
\conv\bigl(\operatorname{MMIX}(q)\bigr)
=\bigcap_{r\in[k]}\conv(\operatorname{MIX}_r({q})).
\end{equation}
where the convex hull of the mixing set is given by the  mixing inequalities \cite{GP01}

\subsection{{An extended formulation for the convex hull of multi-mixing set}}\label{subsec:EF-MMS}
For every $r\in[k]$, choose a permutation $\sigma_r:[n]\to[n]$ such that $q_r^{\sigma_r(1)}
\le
q_r^{\sigma_r(2)}
\le\cdots\le
q_r^{\sigma_r(n)}$.
Set $\eta_{r,0}:=0,$ and $\eta_{r,\ell}:=q_r^{\sigma_r(\ell)},$ for $ \ell\in[n]$.
Define $\Delta_{r,\ell}:=\eta_{r,\ell}-\eta_{r,\ell-1}.$ For each $r\in[k]$, introduce a block $y^r\in\mathbb{R}^n$. The extended formulation of the mixing set $\operatorname{MIX}_r$ is given by \cite{CGZ10}
\begin{equation}\label{eq:mixing-ef}
\widehat{\operatorname{MIX}}_r
:=
\left\{
(u,z,y^r):
\begin{aligned}
z_{\sigma_r(\ell)}
    &\le y^r_\ell
        && \ell\in[n],\\
y^r_{\ell+1}
    &\le y^r_\ell
        && \ell=1,\ldots,n-1,\\
u_r
    &\ge \sum_{\ell=1}^n
        \Delta_{r,\ell}y^r_\ell,\\
0\le z_i&\le1
        && i\in[n],\\
0\le y^r_\ell&\le1
        && \ell\in[n].
\end{aligned}
\right\}.
\end{equation}
Thus $\proj_{u,z}
\bigl(
\widehat{\operatorname{MIX}}_r(q)
\bigr)
=
\operatorname{MIX}_r(q)$.

\begin{theorem}\label{thm:mix-intersection-EF}
The convex hull of the multi-mixing set~\eqref{eq:multi-mixing-set} admits an extended
formulation with $O(kn)$ variables and inequalities.
\end{theorem}

\begin{proof}
By~\eqref{eq:multi-mixing-intersection}, it suffices to show 
\[
\bigcap_{r\in[k]}\operatorname{MIX}_r(q)
=
\proj_{u,z}
\left\{
(u,z,y):
(u,z,y^r)\in
\widehat{\operatorname{MIX}}_r
\quad \forall r\in[k]
\right\}.
\]
We have
\begin{align*}
&(u,z)\in
\proj_{u,z}
\left\{
(u,z,y):
(u,z,y^r)\in
\widehat{\operatorname{MIX}}_r
\quad \forall r\in[k]
\right\}\\
&\quad\Longleftrightarrow
\exists\,y^1,\ldots,y^k
\text{ such that }
(u,z,y^r)\in
\widehat{\operatorname{MIX}}_r
\quad\forall r\in[k]\\
&\quad\Longleftrightarrow
(u,z)\in\operatorname{MIX}_r
\quad\forall r\in[k]\\
&\quad\Longleftrightarrow
(u,z)\in
\bigcap_{r\in[k]}\operatorname{MIX}_r.
\end{align*}
\end{proof}
We use $\widehat{\operatorname{MMIX}}(q)$ to denote the extended formulation of the multi-mixing set using $\widehat{\operatorname{MIX}}_r(q)$.

\subsection{Polynomial-time optimization over multi-mixing sets}
\label{subsec:multi-mixing-algorithm}

Given $c\in\R^n$ and $\lambda\in\R_+^k$, consider the optimization problem
\begin{equation*}
\operatorname{MMOPT}(c,\lambda)
:=\max_{(u,z)\in\operatorname{MMIX}(q)}
\left\{c^\top z-\sum_{r=1}^k\lambda_r u_r\right\}.
\end{equation*}
For fixed binary $z$, the smallest feasible value of $u_r$ is
$\max_{i\in S(z)}q_r^i$, with the maximum over the empty set equal to zero.
Thus, because $\lambda\ge0$,
\begin{equation*}
\operatorname{MMOPT}(c,\lambda)
=\max_{S\subseteq[n]}
\left\{c(S)-\sum_{r=1}^k\lambda_r\max_{i\in S}q_r^i\right\}.
\end{equation*}
We propose a polynomial-time algorithm for this problem using
maximum-weight closure. We first recall this classical graph problem.

\begin{definition}[Maximum-weight closure of a graph]\label{def:closure}
Let $G=(V,A)$ be a directed graph. A set $X\subseteq V$ is a
\emph{closure} of $G$ if $u\in X,\ (u,v)\in A$ implies $v\in X$.
Equivalently, every out-neighbor of a node in $X$ also belongs to $X$.
Moreover, if the nodes are associated with node weights
$\omega\in\R^V$, then the \emph{maximum-weight closure problem} is
$\max\{\sum_{v\in X}\omega_v:X\subseteq V\text{ is a closure of }G\}$.
\end{definition}

The maximum-weight closure problem is polynomially solvable through the
classical reduction to a minimum $s$--$t$ cut; see, e.g.,
\cite{Picard1976,HochbaumQueyranne2003}. For completeness, we recall
the reduction. Let
\[
V^+:=\{v\in V:\omega_v>0\},\qquad
V^-:=\{v\in V:\omega_v<0\},\qquad
W^+:=\sum_{v\in V^+}\omega_v.
\]
Add a source $s$ and a sink $t$. For $v\in V^+$, add an arc $(s,v)$
of capacity $\omega_v$, and for $v\in V^-$, add an arc $(v,t)$ of
capacity $-\omega_v$. For every $(u,v)\in A$, add an arc $(u,v)$ of
capacity $M$, where $M>W^+$. A minimum $s$--$t$ cut never cuts an arc
of capacity $M$, since the cut placing every node of $V$ on the sink
side has capacity $W^+$. Consequently, the original nodes on the
source side form a closure $X$. The capacity of the corresponding cut
is $W^+-\omega(X)$. Thus maximizing the weight of a closure is
equivalent to minimizing the capacity of an $s$--$t$ cut.

We now construct a directed graph $G(c,\lambda)=(V,A)$ whose maximum
closure value is $\operatorname{MMOPT}(c,\lambda)$. Use the permutations
$\sigma_r$ introduced in Section~\ref{subsec:EF-MMS}.
For each $i\in[n]$, create an \emph{item node} $v_i\in V$ with weight
$\omega(v_i):=c_i$.
For every $r\in[k]$ and $\ell\in[n]$, create an \emph{expansion node}
$e_{r,\ell}\in V$ of weight
\begin{equation*}\label{eq:multi-mixing-expansion-weights}
\omega(e_{r,\ell}):=
\begin{cases}
-\lambda_r q_r^{\sigma_r(1)},&\ell=1,\\
-\lambda_r\bigl(q_r^{\sigma_r(\ell)}-q_r^{\sigma_r(\ell-1)}\bigr),
&\ell=2,\ldots,n.
\end{cases}
\end{equation*}
For $\ell=2,\ldots,n$, add the arc
$(e_{r,\ell},e_{r,\ell-1})\in A$.
Finally, for $i=\sigma_r(\ell)$ and every $r\in[k]$, add the arc
$(v_i,e_{r,\ell})\in A$.
Thus, selecting an item node forces all expansion levels needed by
that item in every mixing set. The graph is acyclic.

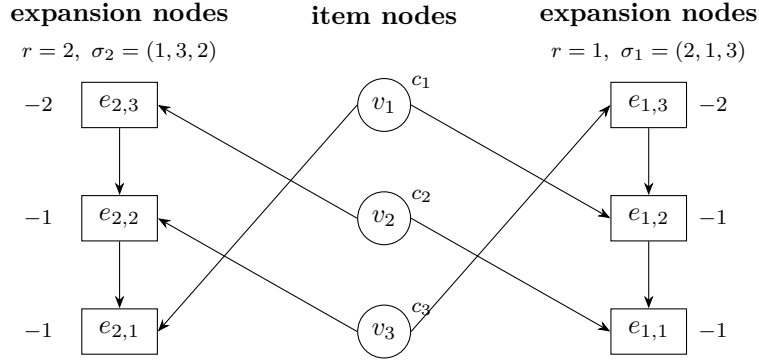
\begin{figure}[t]
\centering
\begin{tikzpicture}[
    >=Stealth,
    item/.style={
        draw, circle, minimum size=7mm,
        inner sep=0pt, font=\small
    },
    exp/.style={
        draw, rectangle, minimum width=10mm,
        minimum height=6mm, inner sep=1pt,
        font=\small
    },
    lab/.style={font=\scriptsize},
    arc/.style={->, thin}
]

\node[font=\small\bfseries] at (-3.5,2.7)
    {expansion nodes};
\node[font=\small\bfseries] at (0,2.7)
    {item nodes};
\node[font=\small\bfseries] at (3.5,2.7)
    {expansion nodes};

\node[lab] at (-3.5,2.2)
    {$r=2,\ \sigma_2=(1,3,2)$};
\node[lab] at (3.5,2.2)
    {$r=1,\ \sigma_1=(2,1,3)$};

\node[item] (v1) at (0,1.5) {$v_1$};
\node[item] (v2) at (0,0) {$v_2$};
\node[item] (v3) at (0,-1.5) {$v_3$};

\node[lab] at (0.5,1.8) {$c_1$};
\node[lab] at (0.5,0.3) {$c_2$};
\node[lab] at (0.5,-1.2) {$c_3$};

\node[exp] (e23) at (-3.5,1.5) {$e_{2,3}$};
\node[exp] (e22) at (-3.5,0) {$e_{2,2}$};
\node[exp] (e21) at (-3.5,-1.5) {$e_{2,1}$};

\node[lab, anchor=east] at (-4.25,1.5) {$-2$};
\node[lab, anchor=east] at (-4.25,0) {$-1$};
\node[lab, anchor=east] at (-4.25,-1.5) {$-1$};

\node[exp] (e13) at (3.5,1.5) {$e_{1,3}$};
\node[exp] (e12) at (3.5,0) {$e_{1,2}$};
\node[exp] (e11) at (3.5,-1.5) {$e_{1,1}$};

\node[lab] at (4.35,1.5) {$-2$};
\node[lab] at (4.35,0) {$-1$};
\node[lab] at (4.35,-1.5) {$-1$};

\draw[arc] (e23.south) -- (e22.north);
\draw[arc] (e22.south) -- (e21.north);

\draw[arc] (e13.south) -- (e12.north);
\draw[arc] (e12.south) -- (e11.north);

\draw[arc] (v1.west) -- (e21.east);
\draw[arc] (v2.west) -- (e23.east);
\draw[arc] (v3.west) -- (e22.east);

\draw[arc] (v1.east) -- (e12.west);
\draw[arc] (v2.east) -- (e11.west);
\draw[arc] (v3.east) -- (e13.west);

\end{tikzpicture}
\caption{Graph-closure reduction for $k=2$ and $n=3$,
with $q_1=(2,1,4)$, $q_2=(1,4,2)$,
and $\lambda_1=\lambda_2=1$.
Each node is labeled with its weight.}
\label{fig:mmix-closure}
\end{figure}

\begin{lemma}\label{lem:multi-mixing-closure}
For every $c\in\R^n$ and $\lambda\in\R_+^k$, the maximum weight of a
closure in $G(c,\lambda)$ is $\operatorname{MMOPT}(c,\lambda)$.
\end{lemma}

\begin{proof}
Let $X\subseteq V$ be a closure, and define its selected item set
by $T(X):=\{i\in[n]:v_i\in X\}$. Fix $r\in[k]$. If $v_i\in X$
and $i=\sigma_r(\ell)$, then the arc $(v_i,e_{r,\ell})$,
together with the chain
\[
e_{r,\ell}\to e_{r,\ell-1}\to\cdots\to e_{r,1},
\]
forces $e_{r,1},\ldots,e_{r,\ell}$ to belong to $X$. Hence the expansion
nodes forced by $T(X)$ in direction $r$ have total weight
\[
-\lambda_r q_r^{\sigma_r(1)}
-\lambda_r\sum_{h=2}^{\ell}
\bigl(q_r^{\sigma_r(h)}-q_r^{\sigma_r(h-1)}\bigr)
=-\lambda_r q_r^{\sigma_r(\ell)}.
\]
Taking $\ell$ to be the largest position of an item in $T(X)$ gives
$-\lambda_r\max_{i\in T(X)}q_r^i$; if $T(X)=\emptyset$, no expansion
nodes are forced. Any additional expansion nodes have nonpositive
weight. Therefore
\[
\omega(X)\le c(T(X))-
\sum_{r=1}^k\lambda_r\max_{i\in T(X)}q_r^i.
\]
This is at most
$\operatorname{MMOPT}(c,\lambda)$ by definition. Taking the maximum over closures
gives one inequality.

For the reverse inequality, fix any $T\subseteq[n]$. We construct a
closure as follows. Select the item nodes $v_i$, $i\in T$. For each
$r\in[k]$, select precisely the prefix of expansion nodes up to the
largest position of an item in $T$, selecting none if $T=\emptyset$.
The resulting node set, denoted $X(T)$, is a closure by construction,
and its weight is
\[
\omega(X(T))=c(T)-\sum_{r=1}^k\lambda_r\max_{i\in T}q_r^i.
\]
Maximizing over $T$ gives the reverse inequality. Hence the maximum
closure weight is exactly $\operatorname{MMOPT}(c,\lambda)$.
\end{proof}

\begin{theorem}[Polynomial-time optimization over multi-mixing sets]
\label{thm:multi-mixing-polytime}
The optimization problem $\operatorname{MMOPT}(c,\lambda)$ can be
solved by one maximum-weight closure computation on a directed acyclic
graph with $O(kn)$ nodes and arcs. Equivalently, it can be solved by
one minimum $s$--$t$ cut computation on a graph of the same order in $O(k^3n^3)$ time plus an overhead of $O(kn\log n)$ time for constructing the graph. 
\end{theorem}
\begin{proof}
By Lemma~\ref{lem:multi-mixing-closure}, the optimal value equals the
maximum closure weight in $G(c,\lambda)$. The graph has $n$ item nodes,
$kn$ expansion nodes, $kn$ item-to-expansion arcs, and $k(n-1)$
chain arcs. Thus it has $O(kn)$ nodes and arcs. The result follows
from the maximum-weight closure/minimum-cut reduction.
\end{proof}

\section{Boxing Problem: Complexity and Extended Formulation}\label{sec:ef}

Fixing an anchor point allows us to model the Boxing Set using a multi-mixing set. We then combine the anchored sets through disjunctive programming to obtain a compact extended formulation for $\conv(B)$.

\subsection{The anchored boxing sets}\label{subsec:anchored-boxing}

Fix an anchor $a\in[n]$. Consider a subset $S\subseteq [n]$ with $a\in S$. Then the width of $S$ is just the total length that the box needs to expand from the singleton point $a$ to include every point in $S$. Thus, its width can be written as:
\[\widthf (S) = \sum_{t = 1}^d \left(
         \max_{\substack{i\in S, \ p^i_t\ge p^a_t}} (p^i_t - p^a_t)
         +\max_{\substack{i\in S, \ p^i_t< p^a_t}} (p^a_t - p^i_t)
        \right)\]
The maximum over an empty index set is taken to be zero.
This decomposition therefore motivates us to define the positive and negative expansions of a point $i$ relative to an anchor $a$:
\begin{equation*}
q_{t,+}^{ai}:=\max\{0,p_t^i-p_t^a\},\qquad
q_{t,-}^{ai}:=\max\{0,p_t^a-p_t^i\},\qquad i\in[n],\ t\in[d].
\end{equation*}
Let $\mathcal H:=\{(t,+),(t,-):t\in[d]\}$ be the set of $2d$ expansion directions, and write $q_h^{ai}$ for the corresponding expansion. For each anchor $a$, define
\[
q^a:=(q_h^{ai})_{h\in\mathcal H,\,i\in[n]}
\in\R_+^{\mathcal H\times[n]},
\qquad q_h^{aa}=0\quad\forall h\in\mathcal H.
\]
Thus, $q^a$ is the coefficient array of the multi-mixing set $\operatorname{MMIX}(q^a)$, with $k=2d$. For every $S\subseteq[n]$ containing $a$, the width is the sum of the expansions from $a$: $\widthf(S)=\sum_{h\in\mathcal H}\max_{i\in S}q_h^{ai}.$

For each $a\in[n]$, define the anchored slice $B_a:=B\cap\{(w,z):z_a=1\}$. The width identity gives
\begin{equation}\label{eq:anchored-multi-mixing-model}
B_a=\left\{(w,z)\in\R\times\{0,1\}^n:\
\begin{array}{l}
\exists\,u\in\R^{\mathcal H}\text{ such that }
(u,z)\in\operatorname{MMIX}(q^a),\\
z_a=1,\quad w\ge\sum_{h\in\mathcal H}u_h
\end{array}
\right\}.
\end{equation}
Using the extended formulation of Section~\ref{subsec:EF-MMS}, define $\widehat B_a
:=\left\{(w,z,u,y):
\begin{aligned}
&(u,z,y)\in\widehat{\operatorname{MMIX}}(q^a),\\
&z_a=1,\qquad w\ge\sum_{h\in\mathcal H}u_h
\end{aligned}
\right\}.$

\begin{lemma}\label{lem:anchored-mix-proj}
For every $a\in[n]$,
\[
\begin{aligned}
\conv(B_a)
&=\left\{(w,z):\begin{array}{l}
\exists\,u\in\R^{\mathcal H}\text{ such that }
(u,z)\in\conv(\operatorname{MMIX}(q^a)),\\
z_a=1,\quad w\ge\sum_{h\in\mathcal H}u_h
\end{array}\right\}\\
&=\proj_{w,z}(\widehat B_a).
\end{aligned}
\]
Moreover, $\widehat B_a$ has $O(dn)$ variables and inequalities.
\end{lemma}

\begin{proof}
By Theorem~\ref{thm:mix-intersection-EF},
$\widehat{\operatorname{MMIX}}(q^a)$ is an extended formulation of
$\conv(\operatorname{MMIX}(q^a))$.
Since $z_a=1$ defines a face of this convex hull, fixing the anchor
preserves the convexification.  Moreover, introducing the continuous variable \(w\) through the linear constraint 
$w\ge\sum_{h\in\mathcal H}u_h$ only adds a free nonnegative slack in $w$ which preserves the convexity. Thus, by~\eqref{eq:anchored-multi-mixing-model},
$\proj_{w,z}(\widehat B_a)=\conv(B_a)$.
The size bound follows from Theorem~\ref{thm:mix-intersection-EF}
with $k=2d$.
\end{proof}

Let $B_0:=\{(w,0)\in\R\times\{0,1\}^n:w\ge0\}$ denote the empty-support slice. The anchored slices give the following decomposition.

\begin{theorem}
\label{thm:boxing-slices}\label{lem:slice}
The Boxing Set satisfies $B=B_0\cup\bigcup_{a=1}^n B_a,$ and $\conv(B)=\conv\left(B_0\cup\bigcup_{a=1}^n B_a\right)$.
Consequently, the Boxing Problem satisfies $\operatorname{BP}(c)=\max\left\{0,\max_{a\in[n]}\max_{(w,z)\in B_a}\{c^\top z-w\}\right\}.$
\end{theorem}

\begin{proof}
Every nonempty support contains an anchor $a\in[n]$, while the empty support belongs to $B_0$. The first equality follows, and taking convex hulls gives the second.
\end{proof}
Since \(q^{aa}_h=0\) for all \(h\in H\) and \(c_a\geq 0\), adding the anchor \(a\) to any selected subset does not increase the expansion cost or decrease the objective value. Hence, the constraint \(z_a=1\) does not affect the optimal value of \(\operatorname{MMOPT}(c,\mathbf{1})\). Therefore, each anchored Boxing Problem is equivalent in optimal value to \(\operatorname{MMOPT}(c,\mathbf{1})\), yielding the following corollary.

\begin{corollary}
\label{cor:boxing-polytime}\label{thm:boxing-polytime}
The Boxing Problem can be solved in $O(d^3n^4)$ time.
\end{corollary}

\subsection{Extended formulation for the Boxing Set}\label{subsec:boxing-EF}

We combine the $n+1$ slices in Theorem~\ref{thm:boxing-slices} using the standard Balas disjunctive formulation. Introduce multipliers $\lambda_a\ge0$, $a=0,\ldots,n$, satisfying $\sum_{a=0}^n\lambda_a=1$. For each $a\in[n]$, introduce a perspective-scaled copy $(w^a,z^a,u^a,y^a)$ of $\widehat B_a$ with multiplier $\lambda_a$. Include $w^0\ge0$ for the empty-support slice.
Let $\mathcal E$ denote the resulting Balas formulation in the variables $(w,z,\xi)$, where $\xi:=(\lambda,w^0,\{w^a,z^a,u^a,y^a\}_{a=1}^n)$
collects all auxiliary variables. Thus $\mathcal E\subseteq\R^{n+1}\times\R^{n+2+n(1+n+2d+2dn)}.$

\begin{corollary}
\label{cor:boxing-EF}\label{thm:boxing-EF}
The formulation $\mathcal E$ satisfies $\proj_{w,z}(\mathcal E)=\conv(B)$. Moreover, $\mathcal E$ has $O(dn^2)$ variables and inequalities.
\end{corollary}
\begin{remark}
    K{\"u}{\c{c}}{\"u}kyavuz~\cite{Kucukyavuz2012} develops a disjunctive extended formulation by fixing possible values of the continuous mixing variable. Direct extension of this strategy to the Boxing Set by enumerating coordinate-wise box extrema may involve up to \(n^{2d}\) disjuncts. In contrast, our anchor-based decomposition requires only \(n+1\) slices and yields an extended formulation of size \(O(dn^2)\), which is polynomial in both \(n\) and \(d\).
\end{remark}

An important consequence of Corollary~\ref{cor:boxing-EF} is that optimization and separation over $\conv(B)$ can be carried out by solving a single polynomial-size linear program.

\begin{corollary}\label{cor:boxing-optimization}
For rational input data, linear optimization over $B$, equivalently over $\conv(B)$, can be performed in time polynomial in $n$, $d$, and the encoding length of the input data.
\end{corollary}

\begin{corollary}\label{cor:boxing-separation}
Separation over $\conv(B)$ can be performed by solving a single linear program of polynomial size. Moreover, for a point $(\hat w,\hat z)$ with $\hat z\in[0,1]^n$, the maximum violation among the facets normalized with coefficient one on $w$, including $w\ge0$, is given by $\min\{w:(w,\hat z,\xi)\in\mathcal E\}-\hat w$.
\end{corollary}

%% file: properties.tex
\section{Basic Polyhedral Properties}\label{sec:basic}
Although the extended formulation gives a compact representation in a
higher-dimensional space, the facets of $\conv(B)$ in the original
$(w,z)$-space have a nontrivial structure which we will explore for the rest of the paper. This section presents several
elementary properties.
\begin{proposition}\label{prop:elementary}The following are basic properties of $\conv(B)$.
    \begin{enumerate}
        \item The convex hull $\conv(B)$ is full-dimensional in $\R^{n+1}$, and
	$\rec(\conv(B))=\{(w,0)\in\R\times\R^n:w\ge0\}.$
        \item The projection of $\conv(B)$ onto the $z$-space is $[0,1]^n$.
	Consequently, the inequalities $0\le z_i \le 1$ for all $i\in [n]$
	are facet-defining for $\conv(B)$, and every facet-defining inequality
	that does not involve $w$ is a multiple of one of these coordinate
	bounds.
    \item $w\ge0$ is facet-defining for $\conv(B)$.
    \end{enumerate}
\end{proposition}
\begin{proposition}\label{prop:normalization}
	Every nontrivial $w$-facet can be written uniquely up to positive
	scaling in the form $w-\alpha^\top z\ge-\beta,$
	where $\alpha\ge0,\beta>0$ and, $\forall i\in[n]$,  $\alpha_i\le\beta$.
\end{proposition}

\begin{proof}
	Write a valid inequality involving $w$ as $\gamma w+c^\top z\ge\delta.$
	Since $(1,0)$ is a recession direction of $\conv(B)$, validity implies
	$\gamma\ge0$.  Because the inequality involves $w$, we have $\gamma>0$
	and may scale it so that $\gamma=1$.
	
	The set $B$ is downward closed in the binary support: if $z',z\in\{0,1\}^n$,
	$z'\le z$, and $(w,z)$ is feasible, then $(w,z')$ is feasible because
	$\widthf(S(z'))\le\widthf(S(z))$.  If some coefficient $c_i>0$, removing
	point $i$ from a tight binary point with $z_i=1$ would decrease the left
	hand side while preserving feasibility, contradicting validity.
	Hence $c_i\le0$ for all $i$.  Set $\alpha:=-c\ge0$ and $\beta:=-\delta$.
	
	Validity at the empty support gives $\beta\ge0$.  If $\beta=0$, validity
	at each singleton gives $-\alpha_i\ge0$, and hence $\alpha=0$, so the
	inequality is $w\ge0$, contrary to nontriviality.  Therefore $\beta>0$.
	Finally, evaluating $w-\alpha^\top z\ge-\beta,$ at $(0,e_i)$ gives
	$\alpha_i\le\beta$.
\end{proof}

We state some properties of the function $wd$ that are crucial for polyhedral analysis.

\begin{proposition}\label{prop:width-properties}
The function $\widthf$ satisfies the following properties.
\begin{enumerate}[label=(\roman*)]
 \item $\widthf(\varnothing)=0$ and $\widthf(\{i\})=0$ for every $i\in[n]$.
 \item $\widthf$ is monotone: if $S\subseteq T$, then $\widthf(S)\le \widthf(T)$.
 \item 
 If $S,T\subseteq[n]$ and $S\cap T\ne\varnothing$, then
 \begin{equation}\label{eq:intersecting-submod}
  \widthf(S\cup T)+\widthf(S\cap T)
  \le
  \widthf(S)+\widthf(T).
 \end{equation}
\end{enumerate}
\end{proposition}

\begin{proof}
The first two properties follow directly from the definition.  For the third, fix $t\in[d]$.  Since $S\cap T\ne\varnothing$,
\[
 \max_{i\in S\cup T}p_t^i
 +
 \max_{i\in S\cap T}p_t^i
 \le
 \max_{i\in S}p_t^i
 +
 \max_{i\in T}p_t^i,
\]
whereas
\[
 \min_{i\in S\cup T}p_t^i
 +
 \min_{i\in S\cap T}p_t^i
 \ge
 \min_{i\in S}p_t^i
 +
 \min_{i\in T}p_t^i.
\]
Subtracting the second inequality from the first gives the coordinate-wise version of~\eqref{eq:intersecting-submod}.  Summing over $t\in[d]$ proves the result.
\end{proof}

\subsection{Uncrossing and laminar structure}\label{sec:laminar}

Fix a nontrivial facet $w-\alpha^\top z\ge-\beta$
of $\conv(B)$.  Define its family of tight sets by
\begin{equation*}
	\mathcal T(\alpha,\beta):=\{S\subseteq[n]:\widthf(S)-\alpha(S)=-\beta\}.
\end{equation*}
For a family $\mathcal A\subseteq2^{[n]}$, let $\widehat{\mathcal A}:=\{(\widthf(S),\ind^S):S\in\mathcal A\}$.
Lemmas~\ref{lem:crossing} and~\ref{lem:laminar-span} are standard results in the uncrossing literature~\cite{LauRaviSingh2011}. We state the lemmas here and provide the proofs in Appendix~\ref{apdx:basic} for completeness.

\begin{restatable}{lemma}{crossing}[Crossing lemma]\label{lem:crossing}
    If $S,T\in\mathcal T(\alpha,\beta)$ and $S\cap T\ne\varnothing$, then $S\cap T\in\mathcal T(\alpha,\beta)$ and $S\cup T\in\mathcal T(\alpha,\beta)$.
	Moreover, $(\widthf(S),\ind^S)+(\widthf(T),\ind^T)
	=(\widthf(S\cap T),\ind^{S\cap T})+(\widthf(S\cup T),\ind^{S\cup T}).$
\end{restatable}

Two sets $S,T\subseteq[n]$ are \emph{compatible} if $S\subseteq T$,
$T\subseteq S$, or $S\cap T=\varnothing$.  A family is \emph{laminar}
if every pair of its members is compatible.  Otherwise, $S$ and $T$
are said to \emph{cross}.

\begin{restatable}{lemma}{laminarSpan}\label{lem:laminar-span}
    Let $\mathcal L$ be an inclusion-wise maximal laminar subfamily of
	$\mathcal T(\alpha,\beta)$.  Then $\aff(\widehat{\mathcal T(\alpha,\beta)})=\aff(\widehat{\mathcal L})$.
\end{restatable}

\begin{corollary}\label{cor:laminar-basis}
	The family $\mathcal L$ contains $n+1$ sets whose corresponding tight
	points are affinely independent.  In particular, the equations
    \begin{equation}\label{eq:laminar-tight-equations}
        \widthf(S)-\alpha(S)=-\beta,\quad S\in\mathcal L,
    \end{equation}
	uniquely determine the coefficients $(\alpha,\beta)$ of the facet.
\end{corollary}

\begin{proof}
	The tight face of a facet of the full-dimensional set
	$\conv(B)\subseteq\R^{n+1}$ has affine dimension $n$.  Hence
	$\widehat{\mathcal T(\alpha,\beta)}$ contains $n+1$ affinely independent
	points.  Lemma~\ref{lem:laminar-span} shows that the same affine hull is
	generated by $\widehat{\mathcal L}$.  Uniqueness of the supporting
	hyperplane follows.
\end{proof}

Associate with $\mathcal L$ a rooted forest
$\mathcal F=(\mathcal L,A)$ as follows.  There is an arc from
$S\in\mathcal L$ to $T\in\mathcal L$ if $T\subsetneq S$ and no member of
$\mathcal L$ lies strictly between $T$ and $S$.  Thus the roots are the
inclusion-wise maximal members of $\mathcal L$, and the children of a
node are its inclusion-wise maximal proper subsets in $\mathcal L$.  Let
$\operatorname{ch}(S)$ denote the union of all children of $S$.

\begin{lemma}\label{lem:leaf-singleton}
	A node $S\in\mathcal L$ is a leaf of $\mathcal F$ if and only if $S$ is
	a singleton.
\end{lemma}

\begin{proof}
	A singleton has no nonempty proper tight subset, and hence is a leaf
	whenever it belongs to $\mathcal L$.  Conversely, suppose that a leaf
	$S$ has at least two elements.  Then choose $v\in \mathbb R^n$ by fixing $j\in S$ and set 
    \begin{equation*}
        v_i = \begin{cases}
            \frac{1}{|S|} & i\neq j, i\in S\\
            -\frac{|S|-1}{|S|} & i = j\\
            0 & i \notin S
        \end{cases}
    \end{equation*}
    Define $\alpha^{(\epsilon)} = \alpha+ \epsilon v$. Consider any set $T\in \mathcal L$. Then either $T$ is disjoint from $S$ or $T$ contains $S$. If $T\cap S = \varnothing$, then $$\widthf(T) - \alpha^{(\epsilon)} (T) = \widthf(T) - \alpha (T) = -\beta.$$ If $S\subseteq T$, then $$\widthf(T) - \alpha^{(\epsilon)} (T) = \widthf(T) - \alpha(T) - \epsilon v(T) = \widthf(T) - \alpha(T)= -\beta.$$
    That is, there is another facet with coefficient $\alpha^{(\epsilon)}$ and $\beta$ that is also tight for all sets in $\mathcal L$, contradicting
	uniqueness in Corollary~\ref{cor:laminar-basis}.
\end{proof}

\begin{lemma}\label{lem:singleton-alpha}
	For $i\in[n]$, the singleton $\{i\}$ belongs to $\mathcal L$ if and only
	if $\alpha_i=\beta$.
\end{lemma}

\begin{proof}
	The singleton is tight if and only if $0-\alpha_i=-\beta,$
	which is equivalent to $\alpha_i=\beta$.  A tight singleton is
	compatible with every set, so maximality of $\mathcal L$ implies that
	it belongs to $\mathcal L$.
\end{proof}

\begin{lemma}\label{lem:one-new-element}
	Every non-leaf node $S\in\mathcal L$ satisfies $|S\setminus\operatorname{ch}(S)|\le1.$
\end{lemma}

\begin{proof}
	The proof follows the same analysis as Lemma~\ref{lem:leaf-singleton}. Let $H:=S\setminus\operatorname{ch}(S)$ and suppose $|H|\ge2$.  Choose
	a nonzero vector $v$ supported on $H$ with $v(H)=0$.  Every member
	$T\in\mathcal L$ is either disjoint from $H$ or contains all of $H$: if
	$T\subsetneq S$, then $T$ lies in one of the children of $S$ and is
	disjoint from $H$; if $S\subseteq T$, then $H\subseteq T$.  Hence
	$v(T)=0$ for every $T\in\mathcal L$.  As in
	Lemma~\ref{lem:leaf-singleton}, this contradicts uniqueness of the
	solution to~\eqref{eq:laminar-tight-equations}.
\end{proof}

\begin{lemma}\label{lem:roots-cover}
	Every element of $[n]$ belongs to a root of $\mathcal F$.
\end{lemma}

\begin{proof}
	If some $j\in[n]$ belonged to no member of $\mathcal L$, then replacing
	$\alpha$ by $\alpha+\eps e_j$ would leave every equation
	\eqref{eq:laminar-tight-equations} unchanged, contradicting
	Corollary~\ref{cor:laminar-basis}.
\end{proof}

\subsection{Sequential lifting}\label{subsec:onedim-sequential-lifting}
Fix $\beta>0$ and a permutation $\pi=(\pi_1,\ldots,\pi_n)$
of $[n]$. Let $A_k:=\{\pi_1,\ldots,\pi_k\}$ be the prefix set indexed by the permutation. Define $A_0 = \varnothing$.
Starting from the valid inequality $w\ge-\beta,$
define recursively
\begin{equation}\label{eq:permutation-lifting-recursion}
\alpha^{\beta,\pi}_{\pi_k}
:=
\min_{S\subseteq A_{k-1}}
\left\{
\widthf(S\cup\{\pi_k\}) -\alpha^{\beta,\pi}(S) +\beta
\right\},
\qquad k=1,\ldots,n.
\end{equation}
For example, the coefficients of the first two variables indexed by the permutation $\pi$ are $\alpha^{\beta,\pi}_{\pi_1}=\beta$
and $\alpha^{\beta,\pi}_{\pi_2}
=
\min\left\{
\beta,\,
|p_{\pi_2}-p_{\pi_1}|
\right\}.$

We remark that~\eqref{eq:permutation-lifting-recursion} is a lifting optimization problem which is the key machinery to yield a valid inequality. Moreover, the lifting problem, after switching minimization to maximization, is exactly the Boxing Problem defined in Section~\ref{sec:intro} for which a polynomial-time algorithm is provided. In this section, we prove that all facets can be obtained via sequential lifting.

Fix an arbitrary facet $w-\bar\alpha^\top z\ge-\bar\beta$, then there exists a laminar forest $\mathcal F(\bar\alpha,\bar\beta)$ associated to the maximal laminar sets family formed from $\mathcal T(\bar\alpha,\bar\beta)$. For each element $j\in[n]$, let $N(j)$ be the unique node in $\mathcal F(\bar\alpha,\bar\beta)$ at which $j$ is introduced. That is, $N(j)\setminus\operatorname{ch}(N(j))=\{j\}.$
\begin{definition}\label{def:bottom-up-order}
    Fix a forest $\mathcal F(\bar\alpha,\bar\beta)$ induced by an arbitrary facet $w-\bar\alpha^\top z\ge-\bar\beta$. A bottom-up ordering induced by $\mathcal F(\bar\alpha,\bar\beta)$ is any permutation $\pi=(\pi_1,\ldots,\pi_n)$ such that, whenever $j$ is introduced at a node $N(j)$, every element contained in the children of $N(j)$ appears before $j$ in $\pi$.
\end{definition}

\begin{theorem}
\label{thm:facet-reconstruction-lifting}
Let $w-\bar\alpha^\top z\ge-\bar\beta$
be a nontrivial facet of $\conv(B)$. Let $\pi$ be any bottom-up ordering induced by $\mathcal F(\bar\alpha,\bar\beta)$. Starting from $w\ge-\bar\beta$
and sequentially lifting the variables in the order $\pi$ using
\eqref{eq:permutation-lifting-recursion} produces exactly the coefficient
vector $\bar\alpha$.
\end{theorem}

\begin{proof}
We induct on the lifting order. Suppose $j = \pi_1$ is the first element in the permutation. Then $j$ is a leaf node whose corresponding variable has coefficient $\bar\alpha_j = \bar \beta$, which agrees with the lifting coefficient.

Now suppose the variables in $A_{k-1}=\{\pi_1,\ldots,\pi_{k-1}\}$
have already received their facet coefficients, and let $j:=\pi_k.$ For every $R\subseteq A_{k-1}$, validity of the final facet
at $R\cup\{j\}$ gives
\[
\widthf(R\cup\{j\})
-\bar\alpha(R)
-\bar\alpha_j
\ge-\bar\beta \iff \widthf(R\cup\{j\})-\bar\alpha(R)+\bar\beta
\ge\bar\alpha_j.
\]
Thus every feasible solution of the lifting problem has value at least
$\bar\alpha_j$.

Conversely, consider $N(j)$, the node introducing $j$, and $\operatorname{ch}(N(j))$. 
Since the node $N(j) =\operatorname{ch}(N(j))\cup\{j\}$ is tight on the facet, we have
\begin{equation}\label{eq:tight-node-lifting-witness}
\widthf(N(j))-\bar\alpha(\operatorname{ch}(N(j)))+\bar\beta
=
\bar\alpha_j.
\end{equation}
Hence $\operatorname{ch}(N(j))\subseteq A_{k-1}$ is a feasible solution to the lifting problem
\eqref{eq:permutation-lifting-recursion}, and the lifting coefficient is
at most $\bar\alpha_j$. Combining this with
\eqref{eq:tight-node-lifting-witness} shows that the exact lifting
coefficient is $\bar\alpha_j$.
\end{proof}

We want to emphasize that Theorem~\ref{thm:facet-reconstruction-lifting} is dimension-independent. However, to obtain all facets using Theorem~\ref{thm:facet-reconstruction-lifting}, we need to enumerate all right-hand side $\beta$. The number of candidate right-hand side varies depending on dimensions. We address this issue in the next two sections.

%% file: 1Dboxing.tex
\section{The One-dimensional Boxing Set}\label{sec:onedim}
Throughout this section, $d=1$.  We assume that the point coordinates are distinct and indexed so that $ p^1<p^2<\cdots<p^n$. The case where the points are non-distinct can be reduced to the distinct case. See Appendix~\ref{subsec:repeated-coordinates} for details. We use throughout the normalization from Proposition~\ref{prop:normalization}. Thus every
nontrivial $w$-facet is written as $w-\alpha^\top z\ge -\beta$, with $\alpha\ge 0$ and $\beta>0$. For a nonempty set $S\subseteq[n]$ we use $u(S)$ and $\ell(S)$ to denote $\max_{i\in S}p^i$ and $\min_{i\in S}p^i$, respectively. Similarly, $\widthf(S) = u(S) - \ell(S)$.

The laminar-family results of Section~\ref{sec:basic} hold without any assumption on the dimension. In one dimension, the geometry of the laminar
forest becomes substantially more explicit, where
disjoint tight sets can be ordered from left to right. Thus, the ordering of the points allows us to turn
this structure into a particularly simple sequential-lifting description of
all facets. 
\subsection{The right-hand side of a one-dimensional facet}
\label{subsec:onedim-right-hand-side}

We now connect the lifting construction to the laminar forest of
Section~4.

Fix a nontrivial facet $w-\alpha^\top z\ge-\beta$
and let $\mathcal L$ be an inclusion-wise maximal laminar family of its tight
sets as in Section~4. For two disjoint sets $S,T\subseteq[n]$, write $S\prec T$ if
$u(S)<\ell(T)$
and, in this case, define
\[
\operatorname{gap}(S,T)
:=
\min_{j\in T}p^j-\max_{i\in S}p^i.
\]

\begin{lemma}\label{lem:disjoint-tight-gap}
Let $S,T\in\mathcal L$ be disjoint. Then their coordinate intervals $[\ell(S), u(S)]$ and $[\ell(T), u(T)]$ do not
overlap. After relabeling so that $S\prec T$, $\operatorname{gap}(S,T)\ge\beta$.
Moreover, $S\cup T$ is tight if and only if $\operatorname{gap}(S,T)=\beta.$
\end{lemma}

\begin{proof}
Since $S$ and $T$ are tight, we have $\alpha(S)=\widthf(S)+\beta,$ and $\alpha(T)=\widthf(T)+\beta$.
Validity at $S\cup T$ gives
\begin{equation}\label{eq:disjoint-tight-validity-bound}
\widthf(S\cup T)-\widthf(S)-\widthf(T)\ge\beta.
\end{equation}
If their coordinate intervals overlapped, the left-hand side of
\eqref{eq:disjoint-tight-validity-bound} would be nonpositive, a
contradiction. Hence the intervals are disjoint.

Assume $S\prec T$. Then $\widthf(S\cup T)
=
\widthf(S)+\widthf(T)+\operatorname{gap}(S,T).$
Substitution in \eqref{eq:disjoint-tight-validity-bound} gives $\operatorname{gap}(S,T)\ge \beta$. Equality in the facet at
$S\cup T$ holds if and only if equality holds in
\eqref{eq:disjoint-tight-validity-bound}, which proves $\operatorname{gap}(S,T)= \beta$.
\end{proof}

The next result is the only place where the one-dimensional geometry is
needed to identify the possible right-hand sides.

\begin{lemma}\label{lem:beta-pair-distance}
For every nontrivial $w$-facet, there exist
indices $a<b$ such that $\beta=p^b-p^a.$
\end{lemma}

\begin{proof}
For a node $S\in\mathcal L$, let $\operatorname{ch}(S)$ denote the union
of its children in the laminar forest. By Lemma~4.9, $|S\setminus\operatorname{ch}(S)|\le1.$
Moreover, every element of $[n]$ is introduced at exactly one node of the
forest: starting from a root containing the element, follow the unique
child containing it until no such child exists.

Thus exactly $n$ nodes of the forest introduce an element. On the other
hand, Corollary~\ref{cor:laminar-basis} implies that $\mathcal L$ contains at least $n+1$
nodes. Hence some non-leaf node $S$ introduces no element. Write its
children, from left to right, as $S_1\prec S_2\prec\cdots\prec S_k$, $k\ge2.$
Then $S=S_1\cup\cdots\cup S_k.$ Since the coordinate intervals are disjoint,
\[
\widthf(S)
=
\sum_{q=1}^k\widthf(S_q)
+
\sum_{q=1}^{k-1}\operatorname{gap}(S_q,S_{q+1}).
\]
Using tightness of $S$ and each $S_q$ gives
\[\sum_{q=1}^{k-1}\operatorname{gap}(S_q,S_{q+1})
=
(k-1)\beta.\]
By Lemma~\ref{lem:disjoint-tight-gap}, every term on the left-hand side is
at least $\beta$. Therefore, every term equals $\beta$. It remains to find two indices $a,b\in [n]$ with $p^b-p^a = \beta$. Choose any $q\in[k-1]$, let
\[
a\in\arg\max\{p^i:i\in S_q\},
\qquad
b\in\arg\min\{p^i:i\in S_{q+1}\}.
\]
Then $p^b-p^a=\operatorname{gap}(S_q,S_{q+1})=\beta$.
\end{proof}

Thus, we define the finite set of candidate right-hand sides using $\text{\ss}
:=
\{p^j-p^i:1\le i<j\le n\}.$
\subsection{Facet description}
\label{subsec:onedim-explicit-facets}

Section~5.1 shows that the right-hand side of every nontrivial
one-dimensional facet belongs to the finite set $\text{\ss}$.
We now make the corresponding facet coefficients explicit. The main observation is that, in one dimension, an exact lifting
coefficient depends only on the nearest variables that have already been
lifted to the left and to the right.

\subsubsection{The nearest-neighbor lifting rule}
Fix $\beta>0$ and a permutation $\pi=(\pi_1,\ldots,\pi_n)$ of $[n]$. For $k\in[n]$, let $A_k:=\{\pi_1,\ldots,\pi_k\}$ be the prefix set, and define $A_0:=\varnothing$. For $k\ge2$, set $j:=\pi_k$ and define the nearest previously inserted
indices to the left and right of $p^j$ by
\[
L_k
:=
\max\{i\in A_{k-1}:i<\pi_k\},
\qquad
R_k
:=
\min\{i\in A_{k-1}:i>\pi_k\},
\]
whenever the corresponding sets are nonempty. We use the notation $(x)_+:=\max\{x,0\}.$
Define $\alpha^{\beta,\pi}\in\mathbb R^n_+$ recursively as follows:
\begin{equation}\label{eq:onedim-neighbor-coefficient}
\alpha^{\beta,\pi}_{\pi_k}
:=
\begin{cases}
\beta & \text{if neither $L_k$ nor $R_k$ exists},\\[1mm]
\min\{\beta,p^{R_k}-p^{\pi_k}\},
    &\text{if $L_k$ does not exist},\\[1mm]
\min\{\beta,p^{\pi_k}-p^{L_k}\},
    &\text{if $R_k$ does not exist},\\[1mm]
\min\!\left\{
    \beta,\,
    p^{\pi_k}-p^{L_k},\,
    p^{R_k}-p^{\pi_k},\,
    (p^{R_k}-p^{L_k}-\beta)_+
    \right\},
    &\text{otherwise}.
\end{cases}\tag{Nearest Neighbor Lifting Rule}
\end{equation}

Thus the coefficient of a newly inserted point depends only on its two
nearest previously inserted neighbors.  We view the lifting process as building the laminar forest of the underlying facet from bottom up. The four quantities in
\eqref{eq:onedim-neighbor-coefficient} have a direct interpretation:
$\beta$ creates a new leaf node, the two one-sided distances extend a current root to the left or to the right, and
$p^{R_k}-p^{L_k}-\beta$ joins two roots.  If
$p^{R_k}-p^{L_k}\le\beta$, the new point receives coefficient zero and extends a current root without growing to either directions.

\begin{example}
Consider the ten points $(p^1,\ldots,p^{10})
=(0,1,2,4,6,7,9,11.5,13,15.5),$
and fix $\beta=4$. We show how to obtain a lifting coefficient using~\ref{eq:onedim-neighbor-coefficient}.
Suppose we have already lifted variables $(2,5,8,1,3,6,9,10)$ and correctly obtained a valid inequality $w-z_1-4z_2-z_3-4z_5-z_6-
4z_8-\tfrac32 z_9-\tfrac52 z_{10}\ge -4.$ Now lift the variable $7$. Its nearest
previously lifted neighbors are $L_k=6$ and
$R_k=8$. By~\ref{eq:onedim-neighbor-coefficient},
\[
\begin{aligned}
\alpha_7=\min\bigl\{
  \beta,\,
  p^7-p^6,\,
  p^8-p^7,\,
  (p^8-p^6-\beta)_+
  \bigr\}=\min\{4,2,2.5,0.5\}
=\tfrac12.
\end{aligned}
\]
Finally, lift variable $4$. Its nearest previously
lifted neighbors are $L_k=3$ and $R_k=5$, with
$p^5-p^3=\beta$. Hence $\alpha_4
=\min\{4,2,2,(6-2-4)_+\}=0.$
\end{example}

\begin{lemma}
\label{lem:onedim-neighbor-lifting}
For every $\beta>0$ and every permutation $\pi$, the coefficient
$\alpha^{\beta,\pi}_{\pi_k}$ in
\ref{eq:onedim-neighbor-coefficient} is exactly the sequential lifting
coefficient
\begin{equation}\label{eq:onedim-exact-lifting}
\min_{S\subseteq A_{k-1}}
\left\{
\widthf(S\cup\{\pi_k\})
-
\alpha^{\beta,\pi}(S)
+
\beta
\right\}.
\end{equation}
Consequently, $w-(\alpha^{\beta,\pi})^\top z\ge-\beta$
is valid for $B$.
\end{lemma}

\begin{proof}
For any integer $k\ge 1$, order the indices in $A_k$ increasingly and partition them whenever two consecutive
selected points are at distance at least $\beta$. The resulting groups will be called the
current $\beta$-components. We first show by induction on $k$ that the lifting following the rules of \eqref{eq:onedim-neighbor-coefficient} preserves the invariant: each of the current $\beta$-component $C$ is tight:
\begin{equation}\label{eq:onedim-component-tight}
\widthf(C)-\alpha^{\beta,\pi}(C)=-\beta.
\end{equation}

This is immediate for $A_1$. Suppose it holds before lifting
$j:=\pi_k$. Consider $A_{k-1}$. Take any $S\subseteq A_{k-1}$. If $S=\varnothing$, the objective in
\eqref{eq:onedim-exact-lifting} equals $\beta$. Now suppose all points of $S$ lie to the left of $j$. By validity of the
current inequality, $\widthf(S)-\alpha^{\beta,\pi}(S)\ge-\beta,$
and therefore $\widthf(S\cup\{j\})
-\alpha^{\beta,\pi}(S)+\beta
\ge
p^j-u(S)
\ge
p^j-p^{L_k}.$
The symmetric argument gives the lower bound
$p^{R_k}-p^j$ when all points of $S$ lie to the right. Finally, suppose $S$ contains points on both sides of $j$. Write
$S^-$ and $S^+$ for its left and right parts. Since $j$ lies between
them,
\[
\widthf(S\cup\{j\})
=
\widthf(S^-)+\widthf(S^+)+\gap(S^-,S^+).
\]
Using validity on $S^-$ and $S^+$ gives
\[
\widthf(S\cup\{j\})
-\alpha^{\beta,\pi}(S)+\beta
\ge
\gap(S^-,S^+)-\beta
\ge
(p^{R_k}-p^{L_k}-\beta)_+.
\]
The first inequality is because adding $j$ to $S$ does not change $\widthf(S)$. The last inequality is by validity.

Thus every feasible $S$ gives a value at least the right-hand side of~
\eqref{eq:onedim-neighbor-coefficient}. Conversely, each term in that
minimum is attained: the value $\beta$ is attained by $S=\varnothing$; the left
distance $p^{\pi_k} - p^{L_k}$ is attained by taking the tight $\beta$-component containing
$L_k$, and the right distance is attained symmetrically. If $L_k$ and
$R_k$ belong to different components, their union attains
$p^{R_k}-p^{L_k}-\beta$; if they belong to the same component, that
component attains zero. Therefore, the~\ref{eq:onedim-neighbor-coefficient} matches the lifting coefficients.

This case analysis also shows that after $j$ is inserted, every new
$\beta$-component continues to satisfy
\eqref{eq:onedim-component-tight}. This completes the induction.
\end{proof}

\subsubsection{Generating facets from $\beta$-blocks}
\label{subsubsec:beta-blocks}

We now give a blockwise description of the facets generated by the
one-dimensional lifting rule.

Fix $\beta\in \text{\ss}$.
Choose indices $b_1<b_2<\cdots<b_m$
satisfying $p^{b_{t+1}}-p^{b_t}\ge\beta,$ for $t=1,\ldots,m-1.$
We call $b_1,\ldots,b_m$ the \emph{$\beta$-leaves}. They induce the
consecutive intervals
\[
[1,b_1],\,
[b_1,b_2],\,
\ldots,\,
[b_{m-1},b_m],\,
[b_m,n].
\]
Adjacent intervals share one $\beta$-leaf. We distinguish
the internal intervals $[b_t,b_{t+1}]$ from the two
exterior intervals $[1,b_1]$ and $[b_m,n]$.

The main idea is to describe the facet coefficients within
each interval through chains (subset) of points. In an \textit{internal
block}, two chains grow inward from the two
$\beta$-leaves. These chains may either remain separated
or be joined by a single point, called the \textit{merger}, giving
rise to a separated block or a merged block, respectively.
In an \textit{exterior block}, only one chain grows outward
from its $\beta$-leaf. The points along these chains
receive positive coefficients determined by their
consecutive coordinate differences, while a merger
receives a coefficient determined by the gap between
the two chains. All remaining points receive coefficient
zero.

We call an interval a \textit{$\beta$-block} if it admits a chain
configuration satisfying the conditions below. We now
specify these conditions and the corresponding facet
coefficients for each block type.

\begin{enumerate}
    \item \textbf{Internal blocks.} Consider interval $I_t=[b_t,b_{t+1}],$ for $t=1,\ldots,m-1.$
$I_t$ is an internal block if there exist two chains:
\begin{equation*}
b_t=r^t_0<r^t_1<\cdots<r^t_{s_t}=c_t
<
d_t=\ell^t_0<\ell^t_1<\cdots<\ell^t_{u_t}=b_{t+1}
\end{equation*}
satisfying
\begin{equation*}
p^{r^t_h}-p^{r^t_{h-1}}<\beta,
\quad h=1,\ldots,s_t,
\qquad
p^{\ell^t_h}-p^{\ell^t_{h-1}}<\beta,
\quad h=1,\ldots,u_t.
\end{equation*}

For a fixed pair of chains satisfying the above conditions,
we distinguish the following two configurations of the chains. If neither configurations hold, the
chosen chains are not admissible for an internal block.

\smallskip
\noindent
\textbf{Separated block.}
The block is separated if either
\begin{equation*}
p^{d_t}-p^{c_t}=\beta,
\end{equation*}
or
\begin{equation*}
d_t=c_t+1
\qquad\text{and}\qquad
p^{d_t}-p^{c_t}>\beta.
\end{equation*}
Its coefficients are 
\begin{equation}\label{eq:separated-block-coefficients}
\alpha_i=
\begin{cases}
p^{r^t_h}-p^{r^t_{h-1}},
    &i=r^t_h,\quad h=1,\ldots,s_t,\\[1mm]
p^{\ell^t_{h+1}}-p^{\ell^t_h},
    &i=\ell^t_h,\quad h=0,\ldots,u_t-1,\\[1mm]
\beta & i = b_t \text{ or } i = b_{t+1},\\[1mm]
0,
    &\text{otherwise in }(b_t,b_{t+1}).
\end{cases}
\end{equation}

\smallskip
\noindent
\textbf{Merged block.}
The block is merged if there exists an index $c_t<\mu_t<d_t$
such that
\begin{equation*}
p^{\mu_t}-p^{c_t}<\beta,
\qquad
p^{d_t}-p^{\mu_t}<\beta,
\qquad
p^{d_t}-p^{c_t}>\beta.
\end{equation*}
The index $\mu_t$ is called the \emph{merger} of the block. Its
coefficients are
\begin{equation}\label{eq:merged-block-coefficients}
\alpha_i=
\begin{cases}
p^{r^t_h}-p^{r^t_{h-1}},
    &i=r^t_h,\quad h=1,\ldots,s_t,\\[1mm]
p^{\ell^t_{h+1}}-p^{\ell^t_h},
    &i=\ell^t_h,\quad h=0,\ldots,u_t-1,\\[1mm]
p^{d_t}-p^{c_t}-\beta,
    &i=\mu_t,\\[1mm]
\beta & i = b_t \text{ or } i = b_{t+1},\\[1mm]
0,
    &\text{otherwise in }(b_t,b_{t+1}).
\end{cases}
\end{equation}
We call $I_t$ a $\beta$-block if it admits at least one
admissible configuration. The chosen chains, and the merger
in the merged case, are part of its description.
\item \textbf{Exterior blocks.} The left exterior block $[1,b_1]$ admits a chain $\ell$: $1=\ell_0<\ell_1<\cdots<\ell_s=b_1$
satisfying $p^{\ell_{h+1}}-p^{\ell_h}<\beta,$ for $h=0,\ldots,s-1$ and its coefficients are
\begin{equation}\label{eq:left-exterior-coefficients}
\alpha_i=
\begin{cases}
p^{\ell_{h+1}}-p^{\ell_h},
    &i=\ell_h,\quad h=0,\ldots,s-1,\\[1mm]
\beta & i = b_1, \\[1mm]
0,
    &\text{otherwise in }[1,b_1).
\end{cases}
\end{equation}

Symmetrically, the right exterior block $[b_m,n]$ admits a chain $r$: $b_m=r_0<r_1<\cdots<r_s=n$
satisfying $p^{r_h}-p^{r_{h-1}}<\beta,$ for $h=1,\ldots,s,$
and its coefficients are 
\begin{equation}\label{eq:right-exterior-coefficients}
    \alpha_i=
\begin{cases}
p^{r_h}-p^{r_{h-1}},
    &i=r_h,\quad h=1,\ldots,s,\\[1mm]
\beta & i = b_{m},\\[1mm]
0,
    &\text{otherwise in }(b_m,n].
\end{cases}
\end{equation}

\end{enumerate}

\begin{remark}
    If an exterior endpoint, $z_1$ or $z_n$, itself has coefficient $\beta$, it is
included among the $\beta$-leaves. In this case, the external block contains only the endpoint.
\end{remark}

\begin{definition}[admissible $\beta$-block system]
\label{def:admissible-beta-block-system}
A collection consisting of $\beta$-leaves $b_1<\cdots<b_m$ and the induced
$\beta$-blocks is called an \emph{admissible $\beta$-block system} if
\begin{enumerate}
    \item $m\ge 2$ and $p^{b_{t+1}}-p^{b_t}\ge\beta$ for $t=1,\ldots,m-1$;
    \item every interval partitioned by the $\beta$-leaves is a $\beta$-block
    satisfying the strict chain and merger conditions above;
    \item at least one internal separated block satisfies $p^{d_t}-p^{c_t}=\beta$.
\end{enumerate}

\end{definition}

\input{Boxing_Beta_Block_Example.tex}

\begin{lemma}
\label{lem:anchor-facet} Let $I_q=[b_q,b_{q+1}]$ be an internal separated
block with $p^{d_q}-p^{c_q}=\beta$.
The inequality $w-\sum_{i\in I_q}\alpha_i z_i\ge-\beta$,
with coefficients given by \eqref{eq:separated-block-coefficients},
defines a facet of $\conv(B[I_q])$.
\end{lemma}

\begin{proof}
Let $r$ and $\ell$ be the two chains in the block, with
$b_q=r_0<\cdots<r_s=c_q<d_q=\ell_0<\cdots<\ell_u=b_{q+1}$.
For $h=0,\ldots,s$, the support $\{r_0,\ldots,r_h\}$ is tight, since
\[
\sum_{v=0}^h\alpha_{r_v}=\beta+p^{r_h}-p^{b_q}
=\beta+\widthf(\{r_0,\ldots,r_h\}).
\]
Similarly, $\{\ell_h,\ldots,\ell_u\}$ is tight for $h=0,\ldots,u$.
Since $p^{d_q}-p^{c_q}=\beta$, the union of the two full chains is
also tight.

These $s+u+3$ tight points are affinely independent in the coordinates
corresponding to the chain variables: their incidence vectors contain
two independent nested chains, and the incidence vector of their union
is the sum of the two full-chain incidence vectors, rather than an
affine combination of them.
For every remaining index $i\in I_q$, the point $p^i$ lies between two
consecutive chain points or between $c_q$ and $d_q$. There is therefore
a tight support containing the corresponding two endpoints. Adding $i$
to this support does not change its width or coefficient sum, and gives
an additional independent tight direction $e_i$.
Hence the tight face contains $|I_q|+1$ affinely independent points,
proving facetness.
\end{proof}

\begin{corollary}\label{cor:block-validity}
    Let $I$ be a $\beta$-block from an admissible $\beta$-block system. Then the inequality $w-\sum_{i\in I}\alpha_i z_i \ge -\beta$, with coefficients $\alpha_i$ specified by its corresponding block type, is valid.
\end{corollary}
Thus every non-anchor internal $\beta$-block consists of a chain growing
from its left $\beta$-leaf, a chain growing from its right $\beta$-leaf,
and possibly one merger joining the two chains.  An internal separated block with central gap $\beta$ may be chosen as an anchor block for proving facetness. The same case analysis as Lemma~\ref{lem:onedim-neighbor-lifting} shows that each $\beta$-block coefficient is obtained by sequentially lifting each chain along its growing direction, followed by lifting the remaining zero-coefficient variables in the $\beta$-block.

\begin{proposition}
\label{lem:beta-block-uniqueness}
The coefficient vector $\alpha$ of an admissible $\beta$-block
system uniquely determines its $\beta$-leaves, the chains in every block,
and the merger, if present.
\end{proposition}

\begin{proof}
Since every chain gap is strictly less than $\beta$, and the merger
coefficient is strictly between $0$ and $\beta$, the $\beta$-leaves are
exactly the indices with $\alpha_i=\beta$. Thus the blocks are uniquely
identified.

Fix an internal block $[b_t,b_{t+1}]$. List its positive-coefficient
interior indices increasingly. Starting at $r_0=b_t$, recover the right
growing chain by taking successive positive indices $r_h$ as long as
$\alpha_{r_h}=p^{r_h}-p^{r_{h-1}}$.
Every index on that chain satisfies this equality. This chain stops at the first index with positive coefficient that fails this equality. In a separated block, this index belongs to the opposite chain, and its coefficient is less
than $\beta\le p^{d_t}-p^{c_t}$. In a merged block this
index is $\mu_t$, and the strict merger condition gives $\alpha_{\mu_t}=p^{d_t}-p^{c_t}-\beta
<p^{\mu_t}-p^{c_t}.$
Hence the right growing chain is uniquely recovered. Symmetrically,
starting at $b_{t+1}$ and moving left recovers the other chain. Any
positive interior index remaining between them is the unique merger.
If none remains, the block is separated. The exterior chains are recovered similarly, starting at their $\beta$-leaves and moving outward. This proves uniqueness.
\end{proof}

\begin{algorithm}[]
\caption{One-dimensional facet generation}
\label{alg:onedim-facet-generation}
\begin{algorithmic}[1]
\Require Distinct coordinates $p^1<\cdots<p^n$.
\For{each distinct $\beta\in$ \ss}
    \For{each admissible $\beta$-block system}
        \State Set $\alpha_{b_t}\gets\beta$ for every $\beta$-leaf $b_t$.
        \For{each internal $\beta$-block $I_t$}
            \If{$I_t$ is a separated block}
                \State Assign its coefficients using
                \eqref{eq:separated-block-coefficients}.
            \Else
                \State Assign its coefficients using
                \eqref{eq:merged-block-coefficients}.
            \EndIf
        \EndFor
        \State Assign the two exterior-block coefficients using
        \eqref{eq:left-exterior-coefficients} and
        \eqref{eq:right-exterior-coefficients}.
        \State Output $w-\alpha^\top z\ge-\beta$.
    \EndFor
\EndFor
\end{algorithmic}
\end{algorithm}
\begin{theorem}\label{prop:beta-block-facets}
Every inequality output by Algorithm~\ref{alg:onedim-facet-generation}
is facet-defining for $\conv(B)$.
\end{theorem}

\begin{proof}
By Definition~\ref{def:admissible-beta-block-system}, at least one internal
block has $p^{d_t}-p^{c_t}=\beta$. Choose this block as the anchor block.
By Lemma~\ref{lem:anchor-facet}, the inequality restricted to the
anchor block is facet-defining.

Consider the following order for the remaining variables. First add the
remaining $\beta$-leaves, moving outward from the anchor block. Then, block
by block, add the variables on each chain from its $\beta$-leaf toward
the center of the block, followed by the merger if present. Finally,
add the zero-coefficient variables.

The~\ref{eq:onedim-neighbor-coefficient}
shows that the $\beta$-leaves receive coefficient $\beta$.
Each chain variable receives its consecutive coordinate difference:
before the merger, the two facing chain endpoints are separated by at
least $\beta$, so the terms involving the opposite chain cannot reduce
this coefficient. A merger receives $p^{d_t}-p^{c_t}-\beta$ by the strict
merger conditions. Every remaining variable lies between already
inserted neighbors at distance at most $\beta$, and hence receives
coefficient zero. Thus every displayed coefficient is an exact
sequential lifting coefficient.
Exact sequential lifting from a local facet preserves facetness, so the
final inequality is facet-defining for $\conv(B)$.
\end{proof}

\subsection{Complete convex hull description}
For this section, we fix an arbitrary nontrivial facet $w-\alpha^\top z\ge-\beta$ of $\conv(B)$. Let $\mathcal L$ be an
inclusion-wise maximal laminar family of its tight sets and $\mathcal F$ be its associated laminar tree. We prove the following lemmas on the structure of the forest $\mathcal F$. A node in $\mathcal F$ is a set in $\mathcal L$ and vice versa. In this section, we refer to a set in $\mathcal L$ and node in $\mathcal F$ interchangeably. Let us begin by exploiting structures on $\mathcal F$.

\subsubsection{Structure of $\mathcal F$}
The following proposition gives a concrete structure on $\mathcal F$. It says that the laminar tree associated to any nontrivial facet is binary. This is the key structure used to prove the completeness of the facet generation procedure, Algorithm~\ref{alg:onedim-facet-generation}. 
\begin{proposition}[One-dimensional node classification]\label{prop:one-d-node-classification}
Every node $S'\in\mathcal F$ is exactly one of the following types.
\begin{enumerate}[label=\textnormal{Type \arabic*.},leftmargin=3.2em]
 \item $S'=\{j\}$, in which case $\alpha_j=\beta$.
 \item $S'=S\cup\{j\}$, where $S$ is the unique child of $S'$, and $\alpha_j$ is given by Lemma~\ref{lem:one-child}.
 \item $S'=S_1\cup S_2\cup\{j\}$, where $S_1\in \mathcal F$ and $S_2\in \mathcal F$ are the two children of $S^{\prime}$, and they satisfy $S_1\prec\{j\}\prec S_2$. $\alpha_j=\operatorname{gap}(S_1,S_2)-\beta.$
 \item $S'=S_1\cup S_2$, where $S_1\in \mathcal F$ and $S_2\in \mathcal F$ are the two children of $S^{\prime}$ and they satisfy $S_1\prec S_2$. $\operatorname{gap}(S_1,S_2)=\beta.$
\end{enumerate}
\end{proposition}
We prove the Proposition assuming Lemma~\ref{lem:binary-node-with-new-element} to Lemma~\ref{lem:one-child}. Proofs of these lemmas are presented in Appendix~\ref{apex:onedim}.

\input{Boxing_Corrected_Laminar_Tree_Example.tex}

\begin{restatable}{lemma}{binaryNodeWithNewElement}\label{lem:binary-node-with-new-element}
    Let $S_1\prec S_2\prec\cdots\prec S_k$
be the children of a node $S'\in\mathcal F$, and suppose $S'
=
\left(\bigcup_{i=1}^k S_i\right)\cup\{j\}.$
Then $k\le2$. If $k=2$, then $S_1\prec\{j\}\prec S_2$ and $\alpha_j=\gap(S_1,S_2)-\beta.$
Moreover, in the latter case, $p^j-u(S_1)<\beta,$ and $\ell(S_2)-p^j<\beta.$
\end{restatable}

\begin{restatable}{lemma}{manyChildreUnion}\label{lem:many-children-union}
    Let $S_1\prec S_2\prec\cdots\prec S_k$ be pairwise disjoint tight sets.  If $S_i\cup S_j$ is not tight for every distinct $i,j$, then
$\bigcup_{i=1}^kS_i$
is not tight whenever $k\ge2$.
\end{restatable}

\begin{restatable}{lemma}{oneChild}\label{lem:one-child}
    If $S$ is the unique child of $S'$ and $S'=S\cup\{j\}$ in $\mathcal F$, then
\[
\alpha_j=\begin{cases}
    p^j-u(S) & \text{if }p^j>u(S),\\
    \ell(S)-p^j & \text{if }p^j<\ell(S),\\
    0 & \text{otherwise}.
\end{cases}
\]
Moreover, $0\le\alpha_j<\beta$.
\end{restatable}

\begin{remark}
    We want to emphasize that a variable $z_j$ gets coefficient $0$ if and only if inserting $j$ into its unique child $S$ does not change the width of the set.
\end{remark}

\begin{proof}[Proof of Proposition~\ref{prop:one-d-node-classification}]
Lemmas~\ref{lem:leaf-singleton} and~\ref{lem:singleton-alpha} show that all leaf node of $\mathcal F$ must be of Type 1. Lemma~\ref{lem:binary-node-with-new-element} shows that if a node introduces a new element, then it has at most 2 children. The same lemma also shows that if this node has exactly 2 children, it is of Type 3. On the other hand, if it has 1 child, then Lemma~\ref{lem:one-child} says it is of Type 2. Finally, Lemma~\ref{lem:many-children-union} implies that when a node does not introduce new element, then it must have exactly 2 children.This proves that all nodes in the laminar forest must be one of the 4 types.
\end{proof}

\begin{corollary}\label{cor:type-four-exists}
The forest $\mathcal F$ contains at least one Type~4 node.
\end{corollary}
\begin{proof}
By Corollary~\ref{cor:laminar-basis}, the laminar family contains at least $n+1$ nodes.  If every node introduced an element not introduced by any descendant, then assigning that element to the node would inject the node set into $[n]$, so the forest would have at most $n$ nodes. Therefore, some non-leaf node introduces no new element.  Proposition~\ref{prop:one-d-node-classification} identifies it as Type~4.
\end{proof}

\subsubsection{Construct an admissible $\beta$-block system from a nontrivial facet}
Section~\ref{sec:basic} shows that all facets are obtained by sequential lifting. We next give an explicit description of these lifting sequences that will eventually lead to a facet. 

\begin{lemma}
\label{lem:beta-component-growth}
Let $w-\bar\alpha^\top z\ge-\bar\beta$ be a nontrivial facet, and choose
a bottom-up ordering of $\mathcal F$ in which all leaf variables
appear first. Between two consecutive $\bar\beta$-leaves, the positive
coefficients arising in this order consist of a chain growing from each
leaf toward the other, and possibly one merger. Every consecutive
chain gap is strictly less than $\bar\beta$. If a merger $\mu_t$
occurs between facing endpoints $c_t<d_t$, then
\[
p^{\mu_t}-p^{c_t}<\bar\beta,\qquad
p^{d_t}-p^{\mu_t}<\bar\beta,\qquad
\bar\alpha_{\mu_t}=p^{d_t}-p^{c_t}-\bar\beta>0.
\]
Every variable subsequently inserted between the joined components
has coefficient zero. The two exterior intervals have one chain each.
\end{lemma}

\begin{proof}
By Theorem~\ref{thm:facet-reconstruction-lifting}, the lifting order
produces exactly $\bar\alpha$. By Lemma~\ref{lem:singleton-alpha}, every
variable inserted after the leaves has coefficient strictly less than
$\bar\beta$.

Consider an index $j$ inserted between two consecutive leaves, and let
$L$ and $R$ be its nearest previously inserted indices. Both exist.
If $p^R-p^L\le\bar\beta$, the nearest-neighbor formula gives
$\bar\alpha_j=0$. Otherwise, if $p^j-p^L<\bar\beta\le p^R-p^j$,
it gives $\bar\alpha_j=p^j-p^L$. This extends the chain growing from
the left. The symmetric case extends the chain growing from the right.
If both $p^j-p^L$ and $p^R-p^j$ are at least $\bar\beta$, the coefficient
would be $\bar\beta$, which is impossible. The only remaining case with
positive coefficient is
\[
p^j-p^L<\bar\beta,\quad p^R-p^j<\bar\beta,
\quad p^R-p^L>\bar\beta.
\]
Then $\bar\alpha_j=p^R-p^L-\bar\beta>0$, and $j$ joins the two
$\bar\beta$-components. After the components join, all consecutive
previously inserted points between the two leaves have gaps strictly
less than $\bar\beta$. Every further insertion in this interval
therefore receives coefficient zero. This also shows there is at most
one merger.

For an index to the left of the first leaf, the one-sided formula
similarly gives either a strict extension of the left exterior chain
or coefficient zero. The right exterior block is symmetric.
\end{proof}

\begin{theorem}
\label{thm:facet-beta-block-decomposition}
Let $w-\bar\alpha^\top z\ge-\bar\beta$
be a nontrivial facet of $\conv(B)$, and let $\mathcal L$ be an
inclusion-wise maximal laminar family of its tight sets. Then the
coefficient vector $\bar\alpha$ is generated by an admissible
$\bar\beta$-block system.
\end{theorem}

\begin{proof}
Let $\mathcal F$ be the laminar forest associated with $\mathcal L$.
We construct an admissible $\bar\beta$-block system from $\mathcal F$.

\paragraph{Find the $\bar\beta$-leaves.} By Lemmas~\ref{lem:leaf-singleton}
and~\ref{lem:singleton-alpha}, the leaves of $\mathcal F$ are precisely
the singletons $\{i\}$ satisfying $\bar\alpha_i=\bar\beta$.
Let their indices, ordered from left to right, be $b_1<\cdots<b_m$.
If $s<t$, validity at the support $\{b_s,b_t\}$ gives
$p^{b_t}-p^{b_s}-2\bar\beta\ge-\bar\beta$, and hence
$p^{b_t}-p^{b_s}\ge\bar\beta$.
Thus, the leaves of $\mathcal F$ form the $\bar\beta$-leaves.

\paragraph{Find the internal blocks.}
Choose a bottom-up lifting order with all leaf variables first.
Such an order exists, and Theorem~\ref{thm:facet-reconstruction-lifting}
shows that it recovers exactly $\bar\alpha$. Fix two consecutive
$\bar\beta$-leaves $b_t<b_{t+1}$.
By Lemma~\ref{lem:beta-component-growth}, the positive-coefficient
variables between these leaves form two chains growing inward, with
strict consecutive gaps, and possibly one merger.

If a positive merger occurs, the block satisfies the strict merged-block
conditions, and its coefficients are given by
\eqref{eq:merged-block-coefficients}. Otherwise, let $c_t<d_t$ be the
facing endpoints immediately before the components first join, if they
join, or their final facing endpoints otherwise. If
$p^{d_t}-p^{c_t}>\bar\beta$, there can be no index strictly between
$c_t$ and $d_t$: otherwise lifting that index would either extend one
of the chains or merge them, since its coefficient is strictly less
than $\bar\beta$. Hence $d_t=c_t+1$. If the components join without
a positive merger, the nearest-neighbor formula forces
$p^{d_t}-p^{c_t}=\bar\beta$. Points inserted across this gap receive
coefficient zero. Thus the block is separated, and its coefficients
are given by \eqref{eq:separated-block-coefficients}.
All remaining variables have coefficient zero, as required.

\paragraph{Find the exterior blocks.}
By Lemma~\ref{lem:beta-component-growth}, each exterior interval
consists of a single outward-growing chain and zero-coefficient
variables. Every chain gap is strictly less than $\bar\beta$, and
its coefficients are given by \eqref{eq:left-exterior-coefficients}
or \eqref{eq:right-exterior-coefficients}.

\paragraph{Choose an anchor block.}
By Corollary~\ref{cor:type-four-exists}, the forest contains a Type~4
node $S=S^-\cup S^+$ with $S^-\prec S^+$. By
Proposition~\ref{prop:one-d-node-classification},
$\gap(S^-,S^+)=\bar\beta$. The rightmost leaf contained in $S^-$
and the leftmost leaf contained in $S^+$ are consecutive
$\bar\beta$-leaves. Indeed, any intervening tight singleton would have
distance at least $\bar\beta$ from each of $S^-$ and $S^+$ by
Lemma~\ref{lem:disjoint-tight-gap}, contradicting their gap being
$\bar\beta$.
For the corresponding block, the two chains cannot merge with positive
coefficient across this gap. Its final facing endpoints therefore have
central gap $\bar\beta$. This gives the required separated block.

Consequently, every interval is a $\bar\beta$-block, at least
one internal block has central gap $\bar\beta$, and the prescribed
coefficient vector is exactly $\bar\alpha$.
\end{proof}

\begin{theorem}
\label{thm:beta-block-bijection}
Assume $p^1<\cdots<p^n$. The nontrivial normalized facets of
$\conv(B)$ are in one-to-one correspondence with the admissible
$\beta$-block systems.
\end{theorem}
\begin{proof}
By Theorem~\ref{prop:beta-block-facets}, each admissible $\beta$-block
system defines a facet. By Theorem~\ref{thm:facet-beta-block-decomposition},
every nontrivial facet arises from such a system.
Finally, Lemma~\ref{lem:beta-block-uniqueness} shows that the coefficient
vector determines the system uniquely. This proves the correspondence.
\end{proof}

\begin{corollary}
\label{thm:onedim-complete-facet-description}
Assume $p^1<p^2<\cdots<p^n.$
Algorithm~\ref{alg:onedim-facet-generation} generates every nontrivial
facet of $\conv(B)$ exactly once. Consequently, $\conv(B)$ is described by $w\ge0,$ and $0\le z_i\le1,$ for all $i\in[n]$,
together with all inequalities output by
Algorithm~\ref{alg:onedim-facet-generation}.
\end{corollary}

\begin{proof}
By Theorem~\ref{prop:beta-block-facets}, every inequality output by
Algorithm~\ref{alg:onedim-facet-generation} is facet-defining.

Conversely, let $w-\bar\alpha^\top z\ge-\bar\beta$
be any nontrivial facet. Theorem~\ref{thm:facet-beta-block-decomposition}
provides a admissible $\bar\beta$-block system whose block
coefficients are exactly $\bar\alpha$. By Lemma~\ref{lem:beta-pair-distance},
$\bar\beta\in\text{\ss}$. Algorithm~\ref{alg:onedim-facet-generation}
enumerates this system and therefore outputs the given facet.
Uniqueness follows from Theorem~\ref{thm:beta-block-bijection}. Proposition~\ref{prop:elementary} accounts for $w\ge0$ and the coordinate-bound facets.
Hence the stated family gives a complete description of $\conv(B)$.
\end{proof}

\subsection{Polynomial-time separation in one dimension}
\label{subsec:onedim-separation}

The lifting description is complete but involves
exponentially many permutations. However, separation
does not require enumerating them. In one dimension,
validity can be checked using only interval supports:
a nontrivial $w$-facet is valid for $\conv(B)$ if and
only if
$\sum_{i=\ell}^r\alpha_i-\beta\le p^r-p^\ell$
for all $1\le\ell\le r\le n$.
Thus, given $(\hat w,\hat z)$ with
$\hat z\in[0,1]^n$, separation over the nontrivial
inequalities of $\conv(B)$ reduces to~\eqref{eq:sep-primal} shown below, where the constant $-\hat w$
is omitted from the primal objective:

\noindent
\begin{minipage}[t]{0.48\linewidth}
\vspace{0pt}
\small
\begin{equation}\label{eq:sep-primal}\tag{P}
\begin{aligned}
\max_{\alpha,\beta}\quad
&\sum_{i=1}^n\hat z_i\alpha_i-\beta\\
\text{s.t.}\quad
&\sum_{i=\ell}^r\alpha_i-\beta
    \le p^r-p^\ell,\\
&\hspace{1em}1\le\ell\le r\le n,\\
&\alpha\ge0,\quad\beta\ge0.
\end{aligned}
\end{equation}
\end{minipage}%
\hfill
\begin{minipage}[t]{0.48\linewidth}
\vspace{0pt}
\small
\begin{equation}\label{eq:sep-dual}\tag{D}
\begin{aligned}
\min_y\quad
&\sum_{\ell\le r}(p^r-p^\ell)y_{\ell r}\\
\text{s.t.}\quad
&\sum_{\ell\le i\le r}y_{\ell r}
    \ge\hat z_i,\quad i\in[n],\\
&\sum_{\ell\le r}y_{\ell r}=1,\\
&y_{\ell r}\ge0,\quad
    1\le\ell\le r\le n.
\end{aligned}
\end{equation}
\end{minipage}
~\eqref{eq:sep-dual} is the corresponding dual. The equality constraint in~\eqref{eq:sep-dual} is because every singleton interval has zero cost, so any unused mass can be assigned to a singleton without changing the objective.

We next give a minimum-cost flow representation of \eqref{eq:sep-dual}.
Let $V=\{s,t\}\cup\{i^-,i^+:i\in[n]\}.$ For every $i\in[n]$, introduce zero-cost arcs $(s,i^-)$, $(i^-,i^+),$ and $(i^+,t),$
and impose the lower bound on flow $f_{i^-i^+}\ge \hat z_i.$
For $i=1,\ldots,n-1$, introduce the arc $(i^+,(i+1)^-)$
with cost $p_{i+1}-p_i$. A total of one unit of $s$-$t$ flow is required. Thus, we obtain the following min-cost flow formulation.
\begin{equation}
\label{eq:mcf}
\begin{aligned}
\min\quad
    & \sum_{i=1}^{n-1}(p_{i+1}-p_i)
      f_{i^+,(i+1)^-}\\
\text{s.t.}\quad
    & f \text{ is an $s$-$t$ flow of value }1,\\
    & f_{i^-i^+}\ge \hat z_i,
        && i\in[n],\\
    & f\ge0.
\end{aligned}
\tag{MCF}
\end{equation}

\begin{lemma}
\label{lem:dual-mcf}
Problems \eqref{eq:sep-dual} and \eqref{eq:mcf} are equivalent and have
the same objective value.
\end{lemma}

\begin{proof}
Every $s$-$t$ path is uniquely associated with an interval $[l,r]$ and
has the form
\[
P_{lr}
=
s\to l^-\to l^+\to\cdots\to r^-\to r^+\to t.
\]
Its cost is $c(P_{lr})
=
\sum_{i=l}^{r-1}(p_{i+1}-p_i)
=
p_r-p_l.$

Given a feasible solution $y$ of \eqref{eq:sep-dual}, send $y_{lr}$
units of flow along the $s$-$t$ path $P_{lr}$. Since $\sum_{l\le r}y_{lr}=1$, this gives
a unit $s$-$t$ flow, satisfying the first constraint. Moreover, $f_{i^-i^+}
=
\sum_{l\le i\le r}y_{lr}
\ge \hat z_i,$ satisfying the second constraint. 
The total cost of the flow is the sum of flow cost along all $s$-$t$ paths. Thus, the objective satisfies 
\[
\sum_{l\le r}y_{lr}c(P_{lr})= \sum_{l\le r}(p_r-p_l)y_{lr}.
\]

Conversely, since the network is acyclic, every feasible unit flow admits
an $s$-$t$ path decomposition. Let $y_{lr}$ be the total flow assigned
to path $P_{lr}$. Then
\[
\sum_{l\le r}y_{lr}=1,
\qquad
\sum_{l\le i\le r}y_{lr}
=
f_{i^-i^+}
\ge \hat z_i,
\]
and the objective values coincide by the expression for $c(P_{lr})$.
\end{proof}

\begin{restatable}{theorem}{onedSeparation}\label{thm:separation-1d}
    When $d=1$, separation over $B$ can be performed in strongly
polynomial time. In particular, after sorting the points, a most
violated normalized nontrivial inequality can be obtained by solving a
minimum-cost flow problem on an acyclic graph in $O(n^2\log^2 n)$ time.
\end{restatable}

%% file: Boxing_Beta_Block_Example.tex

\begin{example}
\label{ex:beta-blocks-ten-points}
Let $\beta=4$ and consider the ten distinct points
\[
(p^1,\ldots,p^{10})=(0,1,2,4,6,7,9,11.5,13,15.5).
\]
Choose the $\beta$-leafs (colored in red) $b_1=2$, $b_2=5$, and $b_3=8$, which
satisfy $p^{b_2}-p^{b_1}=5\ge\beta$ and
$p^{b_3}-p^{b_2}=5.5\ge\beta$. They induce the four blocks
$[1,2]$, $[2,5]$, $[5,8]$, and $[8,10]$. It remains to select appropriate chains to form an admissible $\beta$-block system, illustrated below.

\begin{center}
\begin{adjustbox}{max width=\linewidth}
\begin{tikzpicture}[
    x=0.84cm,y=0.83cm,
    >=Stealth,
    every node/.style={font=\scriptsize},
    leaf/.style={circle,fill=red!70!black,inner sep=2.4pt},
    chain/.style={circle,fill=teal!75!black,inner sep=2.0pt},
    merger/.style={circle,fill=orange!85!black,inner sep=2.5pt},
    zero/.style={circle,fill=gray!55,inner sep=2.0pt},
    chainarrow/.style={-{Stealth[length=3pt]},teal!75!black,semithick},
    mergerarrow/.style={-{Stealth[length=3pt]},orange!85!black,semithick}
]
\draw[gray!60,semithick] (-0.35,0) -- (15.95,0);
\draw[chainarrow] (1,0.94) -- (0,0.94);
\draw[chainarrow] (1,0.94) -- (2,0.94);
\draw[teal!65!black,densely dashed,semithick] (2,0.94) -- (6,0.94);
\node[above,at={(4,0.94)}] {$p^5-p^3=\beta$};
\draw[chainarrow] (6,0.94) -- (7,0.94);
\draw[mergerarrow] (7,0.94) -- (8.90,0.94);
\draw[mergerarrow] (11.5,0.94) -- (9.10,0.94);
\node[above,at={(9,1.18)},text=orange!65!black] {merger $\mu_2=7$};
\draw[chainarrow] (11.5,0.94) -- (13,0.94);
\draw[chainarrow] (13,0.94) -- (15.5,0.94);
\foreach \x/\i/\coord/\coef in {
  0/1/0/1, 1/2/1/4, 2/3/2/1, 4/4/4/0,
  6/5/6/4, 7/6/7/1, 9/7/9/{\tfrac12},
  11.5/8/11.5/4, 13/9/13/{\tfrac32},
  15.5/10/15.5/{\tfrac52}}
{
  \node[above=3pt] at (\x,0) {$\i$};
  \node[below=7pt] at (\x,0) {$\coord$};
  \node[below=24pt] at (\x,0) {$\coef$};
}
\foreach \x in {0,2,7,13,15.5} \node[chain] at (\x,0) {};
\foreach \x in {1,6,11.5} \node[leaf] at (\x,0) {};
\node[merger] at (9,0) {};
\node[zero] at (4,0) {};
\node[anchor=east] at (-0.45,0.38) {index $i$};
\node[anchor=east] at (-0.45,-0.35) {$p^i$};
\node[anchor=east] at (-0.45,-0.89) {$\alpha_i$};
\foreach \a/\b/\desc in {
  0/1/{left exterior},1/6/{separated},
  6/11.5/{merged},11.5/15.5/{right exterior}}
{
  \draw[gray!75,decorate,decoration={brace,mirror,amplitude=4pt}]
       (\a,-1.38) -- (\b,-1.38);
  \node[below] at ({(\a+\b)/2},-1.58) {\desc};
}
\end{tikzpicture}
\end{adjustbox}
\end{center}

The chains and their coefficients are as follows (starting from a $\beta$-leaf, the arrows are pointing to the chain growing direction ). 
\begin{itemize}
\item \emph{Left exterior block $[1,2]$}:
The chain is $\ell=(1,2)$, giving
$\alpha_1=p^2-p^1=1$.
\item \emph{Separated block $[2,5]$}:
Take $r^1=(2,3)$ and $\ell^1=(5)$, with
$c_1=3$ and $d_1=5$. The chain gap is $p^3-p^2=1<\beta$,
while $p^{d_1}-p^{c_1}=p^5-p^3=4=\beta$.
Thus $\alpha_3=1$ and $\alpha_4=0$.
\item \emph{Merged block $[5,8]$}:
Take $r^2=(5,6)$ and $\ell^2=(8)$, with
$c_2=6$, $d_2=8$, and merger $\mu_2=7$. Here
\[
p^7-p^6=2<\beta,\qquad
p^8-p^7=2.5<\beta,\qquad
p^8-p^6=4.5>\beta.
\]
Consequently, $\alpha_6=p^6-p^5=1$ and
$\alpha_7=p^8-p^6-\beta=\tfrac12$.
\item \emph{Right exterior block $[8,10]$:}
The chain is $r=(8,9,10)$, giving
$\alpha_9=p^9-p^8=\tfrac32$ and
$\alpha_{10}=p^{10}-p^9=\tfrac52$.
\end{itemize}
The three $\beta$-leafs have
$\alpha_2=\alpha_5=\alpha_8=\beta=4$.
Thus the resulting coefficient vector is
\[
\alpha=(1,4,1,0,4,1,\tfrac12,4,\tfrac32,\tfrac52),
\]
and the corresponding facet is $ w-z_1-4z_2-z_3-4z_5-z_6-\tfrac12 z_7
 -4z_8-\tfrac32 z_9-\tfrac52 z_{10}\ge-4.$
The separated block $[2,5]$ is an anchor block because
$p^5-p^3=\beta$. The admissible $\beta$-block system
therefore defines a facet of $\conv(B)$.
\end{example}

%% file: Boxing_Corrected_Laminar_Tree_Example.tex
\begin{example}
Consider the points $(p^1,\ldots,p^{10})=(0,1,2,4,6,7,9,11.5,13,15.5)$
and the facet $ w-z_1-4z_2-z_3-4z_5-z_6-\tfrac12z_7
-4z_8-\tfrac32z_9-\tfrac52z_{10}\ge -4.$
The following laminar forest (tree) consists of tight sets for this facet.
Each node is labeled by its type in Proposition~\ref{prop:one-d-node-classification}.

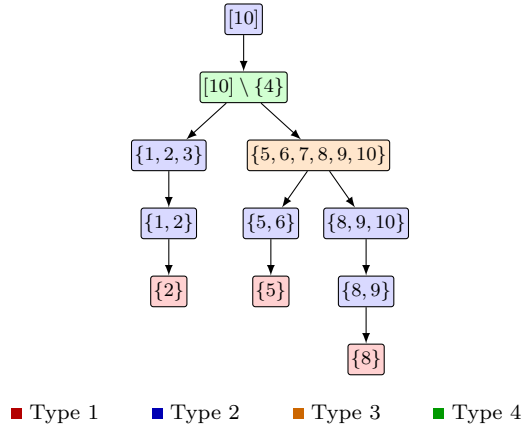
\begin{figure}[H]
\centering
\begin{tikzpicture}[
  scale=0.9,
  transform shape,
  level 1/.style={sibling distance=38mm, level distance=10mm},
  level 2/.style={sibling distance=22mm, level distance=10mm},
  level 3/.style={sibling distance=14mm, level distance=10mm},
  level 4/.style={sibling distance=10mm, level distance=10mm},
  every node/.style={
    draw,
    rounded corners=1pt,
    font=\scriptsize,
    inner sep=1.2pt,
    minimum height=4.5mm,
    align=center
  },
  edge from parent/.style={draw, -latex, line width=0.4pt},
  type1/.style={fill=red!18},
  type2/.style={fill=blue!15},
  type3/.style={fill=orange!20},
  type4/.style={fill=green!18}
]
\node[type2] {$[10]$}
  child {
    node[type4] {$[10]\setminus\{4\}$}
      child {
        node[type2] {$\{1,2,3\}$}
          child {
            node[type2] {$\{1,2\}$}
              child {
                node[type1] {$\{2\}$}
              }
          }
      }
      child {
        node[type3] {$\{5,6,7,8,9,10\}$}
          child {
            node[type2] {$\{5,6\}$}
              child {
                node[type1] {$\{5\}$}
              }
          }
          child {
            node[type2] {$\{8,9,10\}$}
              child {
                node[type2] {$\{8,9\}$}
                  child {
                    node[type1] {$\{8\}$}
                  }
              }
          }
      }
  };
\end{tikzpicture}

\vspace{1mm}
{\scriptsize
\textcolor{red!70!black}{\rule{1.2ex}{1.2ex}} Type 1 \qquad
\textcolor{blue!70!black}{\rule{1.2ex}{1.2ex}} Type 2 \qquad
\textcolor{orange!80!black}{\rule{1.2ex}{1.2ex}} Type 3 \qquad
\textcolor{green!60!black}{\rule{1.2ex}{1.2ex}} Type 4
}

\caption{Laminar forest associated with
$w-z_1-4z_2-z_3-4z_5-z_6-\tfrac12 z_7-4z_8-\tfrac32 z_9-\tfrac52 z_{10}\ge -4$.}
\label{fig:laminar-tree-example}
\end{figure}

\end{example}

%% file: higher_dim.tex
\section{The Convex Hull in Higher Dimensions}\label{sec:higher-dim}

The uncrossing and laminar-basis results of Section~\ref{sec:basic} hold in every dimension. What fails beyond one dimension is the additive left-to-right gap structure used to classify branching nodes.

For the rest of this section, we fix an arbitrary nontrivial facet $w-\alpha^\top z\ge-\beta$ of $\conv(B)$. Let $\mathcal L$ be an
inclusion-wise maximal laminar family of its tight sets and $\mathcal F$ be its associated laminar forest.

\subsection{The failure of the one-dimensional gap decomposition}

For pairwise disjoint sets $S_1,\ldots,S_k \in \mathcal L$, define
\begin{equation}\label{eq:aggregate-gap}
 \Gamma(S_1,\ldots,S_k)
 :=
 \widthf\left(\bigcup_{i=1}^kS_i\right)
 -\sum_{i=1}^k\widthf(S_i).
\end{equation}
In one dimension, ordered disjoint sets satisfy
\[
 \Gamma(S_1,\ldots,S_k)
 =
 \sum_{i=1}^{k-1}\operatorname{gap}(S_i,S_{i+1}).
\]
There is no analogous pairwise decomposition in higher dimensions. Consider the example: $ p_1=(0,0),\  p_2=(0,2),\  p_3=(1,1),$
each relevant pair has total width two, whereas $ \Gamma(\{1\},\{2\},\{3\})=3.$
There is no ordering of these three singleton sets for which the total width equals the sum of two consecutive pairwise gaps.

The aggregate quantity~\eqref{eq:aggregate-gap} has a direct interpretation for tight nodes.

\begin{lemma}\label{lem:multi-gap-tight}
Let $S_1,\ldots,S_k$ be pairwise disjoint tight sets for a nontrivial facet $w-\alpha^\top z\ge -\beta$.  If $\bigcup_{i=1}^kS_i$ is also tight, then $\Gamma(S_1,\ldots,S_k)=(k-1)\beta.$
\end{lemma}

\begin{proof}
Since the sets are disjoint,
\[
 \alpha\left(\bigcup_{i=1}^kS_i\right)
 =\sum_{i=1}^k\alpha(S_i)
 =\sum_{i=1}^k\widthf(S_i)+k\beta.
\]
Tightness of the union gives
\[
 \widthf\left(\bigcup_{i=1}^kS_i\right)
 -\sum_{i=1}^k\widthf(S_i)-k\beta
 =-\beta,
\]
which is the desired identity.
\end{proof}

\subsection{General node configurations}

Let $S'\in\mathcal F$ be a node with children $S_1,\ldots,S_k$.  Lemma~\ref{lem:one-new-element} implies that $S'$ contains either no new element or one new element beyond the union of its children.

\begin{proposition}[Higher-dimensional node classification]\label{prop:hd-node-classification}
Every node $S'\in\mathcal F$ is of one of the following types.
\begin{enumerate}[label=\textnormal{Type \arabic*.},leftmargin=3.5em]
 \item $S'=\{j\}$.  Then $\alpha_j=\beta$.
 \item $S'=S\cup\{j\}$, where $S$ is the unique child. Then $\alpha_j=\widthf(S')-\widthf(S).$
 \item $ S'=\left(\bigcup_{i=1}^kS_i\right)\cup\{j\},$
 where $S_1,\ldots,S_k$ are the children of $S'$ in $\mathcal F$. Then
 \begin{equation}\label{eq:hd-type3}
  \alpha_j
  =
  \widthf(S')-
  \sum_{i=1}^k\widthf(S_i)
  -(k-1)\beta.
 \end{equation}
 \item $S'=\bigcup_{i=1}^kS_i,\ k\ge 2$, 
 where $S_1,\ldots,S_k$ are the children of $S'$ in $\mathcal F$. Then
$ \Gamma(S_1,\ldots,S_k)=(k-1)\beta.$
\end{enumerate}
\end{proposition}
\begin{remark}
    Type~2 is the special case $k=1$ of Type~3.
\end{remark}

\begin{proof}
The structural alternatives follow from Lemmas~\ref{lem:leaf-singleton} and~\ref{lem:one-new-element}.  Type~1 follows from Lemma~\ref{lem:singleton-alpha}.  In Type~3, tightness of $S'$ and of its children gives
\begin{align*}
 -\beta
 &=\widthf(S')-\alpha_j-
   \sum_{i=1}^k\alpha(S_i)=\widthf(S')-\alpha_j-
   \sum_{i=1}^k[\widthf(S_i)+\beta],
\end{align*}
which yields~\eqref{eq:hd-type3}.  Type~4 is Lemma~\ref{lem:multi-gap-tight}.
\end{proof}

The laminar family contains at least $n+1$ nodes. Every element is introduced at exactly one node, and
each node introduces at most one element. Hence some node introduces no element. Such a node is of
Type~4. 
\subsection{Characterization of the right-hand-side of a nontrivial facet}
Type 4 nodes give a characterization of the right-hand sides of nontrivial facets through partitions of tight sets. This suggests that to find all possible right-hand-sides, we need to consider partitions of subsets \(S\subseteq[n]\). In this section, we
identify a parameter $\beta_S^*$ for every subset \(S\) with \(|S|\geq 2\), and show that a partition
of $S$ gives the right-hand-side of a facet if and only if it attains
$\beta_S^*$. Moreover, every nontrivial normalized facet has right-hand-side parameter $\beta_S^*$ for some such $S$.

For any $S\subseteq[n]$ with $|S|\ge2$, let
$\Pi(S)$ denote the collection of partitions of $S$.
Given a partition
$\mathcal P=\{S_1,\ldots,S_k\}\in\Pi(S)$
with $k\ge2$, define
\[
\beta_S(\mathcal P)
=
\frac{
\operatorname{wd}(S)
-\sum_{i=1}^k\operatorname{wd}(S_i)
}{k-1},
\]
and let $\beta_S^*
=
\max_{\substack{\mathcal P\in\Pi(S)\\
                 |\mathcal P|\ge2}}
\beta_S(\mathcal P).$

\begin{lemma}\label{lem:admissible-partition}
Let $\mathcal P=\{S_1,\ldots,S_k\}$ be a partition of
$S\subseteq[n]$ with $k\ge2$.
Suppose that $w-\alpha^\top z\ge-\beta$ is valid for
$\conv(B)$ and that $S,S_1,\ldots,S_k$ are all tight.
Then there exists a facet-defining inequality with the
same right-hand side for which these sets remain tight.
\end{lemma}

\begin{proof}
    Consider the normalized polar of $\conv(B)$: $Q:= \{(\alpha,\beta)\in\mathbb R^{n+1}: \alpha (T) -\beta \le \widthf(T),\ \forall T\subseteq [n] \}$. Now define a face of the polar using the tight sets $S, S_1,... S_k$:
    $F_Q = \{(\alpha, \beta)\in Q: \alpha(T)- \beta = \widthf(T),\ \forall T\in \{S,S_1,...,S_k\}\}$. 
    The polyhedron $Q$ is pointed, since any lineality
direction $(h,\tau)$ must satisfy
$h(T)-\tau=0$ for all $T\subseteq[n]$. Consider
taking $T=\varnothing$ and $T=\{i\}$ gives
$\tau=0$ and $h=0$.
Since $F_Q$ is a nonempty face of the pointed
polyhedron $Q$, it contains a vertex, which
is also a vertex of $Q$, say $(\alpha^*,\beta^*)$.Thus the inequality $w- \alpha^{*\top} z\ge -\beta^*$ is a facet of $\conv B$ and is tight on $S, S_1,... S_k$. Since $S_1,\ldots,S_k$ partition $S$,
summing their tight equations and comparing
with the tight equation for $S$ gives
\[
(k-1)\beta^*
=\widthf(S)-\sum_{i=1}^k\widthf(S_i)
=(k-1)\beta.
\]
Hence $\beta^*=\beta$, as desired.
\end{proof}
Lemma~\ref{lem:admissible-partition} motivates the following definition. 
\begin{definition}[Facet admissible partition]
Given a subset $S\subseteq[n]$, a partition $\mathcal P = \{S_1,..., S_k\}$ of $S$, where $k \ge 2$, is called \textit{facet-admissible} if there exists a facet $w-\alpha^\top z \ge - \beta $ of $\conv(B)$ such that $S_1, S_2,.., S_k$ and $S$ are all tight.
\end{definition}

\begin{lemma}\label{cor:admissible-part-beta}
    Suppose a valid inequality $w-\alpha^\top z \ge - \beta$ is tight on the set $S$. Then $\beta_S(\mathcal P)\le \beta$ for every $\mathcal P\in \Pi(S)$ with $|\mathcal P|\ge 2$, with equality if and only if every set in $\mathcal P$ is tight for the inequality. 
\end{lemma}
\begin{proof}
Let $\mathcal P=\{S_1,\ldots,S_k\}$ be
any partition of $S$ with $k\geq 2$.
By validity,
\[
\alpha(S_i)\leq \widthf(S_i)+\beta,
\qquad i=1,\ldots,k.
\]
Since $S$ is tight,
\[
\widthf(S)+\beta
=\alpha(S)
=\sum_{i=1}^k\alpha(S_i)
\leq\sum_{i=1}^k(\widthf(S_i)+\beta).
\]
Rearranging gives
\[
\beta_S(\mathcal P)
=\frac{\widthf(S)-\sum_{i=1}^k\widthf(S_i)}{k-1}
\leq\beta.
\]
Equality holds if and only if every
$S_i\in\mathcal P$ is also tight.
\end{proof}

The reverse of the lemma is also true. We provide the proof in Appendix~\ref{subsec:computing-beta}. 
\begin{restatable}{theorem}{allSubsetAdmissible}
\label{thm:all-subsets-admissible}
For every $S\subseteq[n]$ with $|S|\ge2$,
every partition $\mathcal P\in\Pi(S)$
attaining $\beta_S^*$ is facet-admissible.
Consequently, every such subset $S$
admits a facet-admissible partition.
\end{restatable}
\begin{remark}
    If $\widthf(S)>0$, then $\beta_S^*>0$, and the facet
in Theorem~\ref{thm:all-subsets-admissible} is nontrivial.
If $\widthf(S)=0$, then $\beta_S^*=0$, and all partitions
of $S$ are facet-admissible through the facet $w\ge0$.
\end{remark}

\begin{corollary}
Let $S\subseteq[n]$ with $|S|\geq 2$.
A partition $\mathcal P\in\Pi(S)$ with
$|\mathcal P|\geq 2$ is facet-admissible
if and only if
$\beta_S(\mathcal P)=\beta_S^*$. We have the following equality:
\begin{align*}
    &\{\beta>0: w-\alpha^\top z \ge -\beta \text{ defines a nontrivial facet for some }\alpha\} \\&= \{\beta^*_S:
S\subseteq [n], |S|\ge 2, \widthf(S)>0\}
\end{align*}
\end{corollary}

The preceding results establish that the facet-admissible partitions of a given subset (S) are exactly those attaining \(\beta_S^*\). We next provide a polynomial-time algorithm for computing \(\beta_S^*\) and a maximizing partition. The proof is given in Appendix~\ref{subsec:computing-beta}.

\begin{restatable}{theorem}{betaPolynomial}\label{thm:beta-polynomial}
    There exists an algorithm that computes
$\beta_S^*$ and a maximizing partition for every $S\subseteq[n]$ with $|S|\ge2$ using at most $O(|S|^2)$
submodular-function minimizations.
In particular, the algorithm is strongly
polynomial in the value-oracle model.
\end{restatable}

\subsection{Family of facets}

We next provide a family of facets by identifying the smallest value of $\beta$ arising in normalized nontrivial facets of $\conv(B)$.

\begin{proposition}\label{prop:beta-min}
Assume that no two points coincide.  Then
\begin{equation}\label{eq:beta-min}
 \beta_{\min}
 :=
 \min_{\substack{S\subseteq[n]\\|S|\ge2}}
 \frac{\widthf(S)}{|S|-1}
\end{equation}
is the smallest possible right-hand side of a nontrivial normalized facet of $\conv(B)$.  Moreover,
\begin{equation}\label{eq:beta-min-facet}
 w-\beta_{\min}\sum_{i=1}^n z_i\ge-\beta_{\min}
\end{equation}
is facet-defining.
\end{proposition}

\begin{proof}
Let $w-\alpha^\top z\ge-\beta$ be a nontrivial facet of $\conv(B)$.  Proposition~\ref{prop:normalization} gives $\alpha_i\le\beta$.  Since a facet has $n+1$ affinely independent tight points and there are only $n$ singleton supports, some tight set $T$ satisfies $|T|\ge2$.  Tightness gives
\[
 \widthf(T)=\alpha(T)-\beta
 \le(|T|-1)\beta,
\]
and hence
\[
 \beta\ge\frac{\widthf(T)}{|T|-1}\ge\beta_{\min}.
\]

For validity of~\eqref{eq:beta-min-facet}, let $S\subseteq[n]$.  The empty set and singleton sets satisfy the inequality at equality or slack.  If $|S|\ge2$, definition~\eqref{eq:beta-min} gives $ \widthf(S)\ge(|S|-1)\beta_{\min},$
which is equivalent to validity.  Every singleton is tight, and any minimizer $S^*$ in~\eqref{eq:beta-min} is tight.  The $n$ singleton points together with $(\widthf(S^*),\ind^{S^*})$ are affinely independent, because $\ind^{S^*}$ does not lie in the affine hyperplane $\sum_i z_i=1$ containing the singleton incidence vectors.  Thus~\eqref{eq:beta-min-facet} is facet-defining.
\end{proof}
We propose an algorithm that finds $\beta_{\min}$ in polynomial time. The proof is deferred to Appendix~\ref{sec:findBetaMin}.
\begin{restatable}{theorem}{polytimeBetaMin}\label{thm:polytimeBetaMin}
There exists an algorithm that computes
$\beta_{\min}$ and a subset attaining
\eqref{eq:beta-min} using at most $O(n^2)$
submodular-function minimizations.
In particular, the algorithm is strongly
polynomial in the value-oracle model, even
when the dimension $d$ is part of the input.
\end{restatable}

%% file: discussion.tex
\section{Discussion}
\begin{enumerate}
    \item Many arguments in this paper, such as the uncrossing argument, are independent of the $\widthf$ function we are using. Thus, these arguments can be applied to a broader class of set functions.
    \item We are able to give an explicit characterization of the right-hand sides of all facets in the higher-dimensional case. We are hoping to give an explicit description for all facets, not just the right-hand sides. 
\end{enumerate}

%% file: appendix.tex
\section{Omitted Proofs in Section~\ref{sec:basic}}\label{apdx:basic}

\crossing*

\begin{proof}
	Let $\phi(A):=\widthf(A)-\alpha(A)$.  By
	Proposition~\ref{prop:width-properties} and modularity of $\alpha$,
	\[
	\phi(S\cup T)+\phi(S\cap T)\le\phi(S)+\phi(T)=-2\beta.
	\]
	Validity gives $\phi(S\cup T)\ge-\beta$ and $\phi(S\cap T)\ge-\beta.$
	Both inequalities must therefore hold at equality, proving that the
	union and intersection are tight.  It also follows that \[\widthf(S\cup T) + \widthf(S\cap T) = \widthf(S)+ \widthf(T) \]  The incidence-vector identity $\ind^S+\ind^T=\ind^{S\cap T}+\ind^{S\cup T}$
	then proves the stated equality.
\end{proof}

\laminarSpan*

\begin{proof}
	For $S\in\mathcal T(\alpha,\beta)$, define
	\[
	\cL(S):=\{R\in\mathcal L:R\text{ crosses }S\}.
	\]
	We prove by induction on $|\cL(S)|$ that
	$(\widthf(S),\ind^S)\in\aff(\widehat{\mathcal L})$.
	
	If $\cL(S)=\varnothing$, then $S$ is compatible with every member of
	$\mathcal L$.  Maximality of $\mathcal L$ gives $S\in\mathcal L$.
	
	Now suppose $\cL(S)\ne\varnothing$ and choose $R\in\cL(S)$.  By
	Lemma~\ref{lem:crossing}, both $S\cap R$ and $S\cup R$ are tight, and
	\[
	(\widthf(S),\ind^S)
	=(\widthf(S\cap R),\ind^{S\cap R})
	+(\widthf(S\cup R),\ind^{S\cup R})
	-(\widthf(R),\ind^R).
	\]
	Moreover,
	\[
	\cL(S\cap R)\subsetneq\cL(S),\qquad
	\cL(S\cup R)\subsetneq\cL(S).
	\]
	Indeed, if $Q\in\mathcal L$ crosses $S\cap R$, laminarity of $Q$ and $R$
	forces $Q\subsetneq R$, and hence $Q$ crosses $S$.  A similar case
	distinction shows that every member of $\mathcal L$ crossing $S\cup R$
	also crosses $S$.  Finally, $R$ crosses $S$ but crosses neither
	$S\cap R$ nor $S\cup R$, so both inclusions are strict.
	
	The induction hypothesis therefore applies to $S\cap R$ and $S\cup R$,
	while $R\in\mathcal L$.  Hence
	$(\widthf(S),\ind^S)\in\aff(\widehat{\mathcal L})$.  The reverse
	inclusion follows from $\mathcal L\subseteq\mathcal T(\alpha,\beta)$.
\end{proof}

\section{Omitted Proofs in Section~\ref{sec:onedim}}\label{apex:onedim}

\binaryNodeWithNewElement*

\begin{proof}
Before proving the statement, we first show that the union of two consecutive children is not tight.
For $i=1,\ldots,k-1$, define $U_i:=S_i\cup S_{i+1}.$
We claim that $U_i$ is compatible with every member of $\mathcal L$.

Indeed, let $R\in\mathcal L$. Since $R$ and $S'$ belong to the laminar
family, either $R\cap S'=\emptyset$, $S'\subseteq R$, or $R\subsetneq S'$.
The first two cases clearly make $R$ compatible with $U_i$. In the last
case, since $S_1,\ldots,S_k$ are the children of $S'$, the set $R$ is
contained in one of these children. Hence either $R\subseteq U_i$ or
$R\cap U_i=\emptyset$. Thus $U_i$ is compatible with every member of
$\mathcal L$.

If $U_i$ were tight, maximality of $\mathcal L$ would therefore imply
$U_i\in\mathcal L$. But $S_i\subsetneq U_i\subsetneq S',$
contradicting the fact that $S_i$ is a child of $S'$. Consequently,
$U_i$ is not tight. By Lemma~\ref{lem:disjoint-tight-gap}, for every $i=1,\ldots,k-1$, we have $\gap(S_i,S_{i+1})>\beta.$ We use this key result to prove the Lemma.

\paragraph{Show $k\le 2$.} To this end, we use validity after deleting an extreme child: $S_1$ or $S_k$.

Suppose for the contrary that $\{j\}\not \prec S_2$ , then $p^j\ge\ell(S_2).$
Let $T:=S'\setminus S_1.$
Because $S_1$ is the leftmost child and $p^j\ge\ell(S_2)$, deleting $S_1$
moves the left endpoint from $\ell(S_1)$ to $\ell(S_2)$ and does not change the
right endpoint. Hence
\[
\widthf(S')-\widthf(T)
=
\ell(S_2)-\ell(S_1)
=
\widthf(S_1)+\gap(S_1, S_2).
\]
Since $S_1$ is tight, $\alpha(S_1)=\widthf(S_1)+\beta.$
Using tightness of $S'$ and the identity $\alpha(T)=\alpha(S')-\alpha(S_1),$
we obtain
\begin{align*}
\widthf(T)-\alpha(T)
&=
\widthf(S')-\alpha(S')
-\bigl(\widthf(S')-\widthf(T)\bigr)
+\alpha(S_1)\\
&=
-\beta
-\bigl(\widthf(S_1)+\gap(S_1, S_2)\bigr)
+\bigl(\widthf(S_1)+\beta\bigr)\\
&=
-\gap(S_1, S_2)\\
&< -\beta.
\end{align*}
which contradicts the validity of the facet at the support $T$. Therefore $\{j\}\prec S_2$. A symmetric argument applied to $T':=S'\setminus S_k$
shows that $\{j\}\succ S_{k-1}.$

If $k\ge3$, then $\ell(S_2)\le u(S_{k-1}),$
whereas the above arguments require
\begin{equation}\label{eq:k>=3-contradiction}
    p^j<\ell(S_2)
\le
u(S_{k-1})
<p^j,
\end{equation}
a contradiction. Hence $k\le2.$

\paragraph{Show $S_1\prec\{j\}\prec S_2$ and $\alpha_j=\gap(S_1,S_2)-\beta.$ when $k =2$.} Suppose now that $k=2$. Then the inequality~\eqref{eq:k>=3-contradiction} become $u(S_1)<p^j<\ell(S_2),$
and therefore $S_1\prec\{j\}\prec S_2.$

Because $p^j$ lies strictly between the two children, adding $j$ does not
change the width of their union. Hence
\[
\widthf(S')
=
\widthf(S_1)+\gap(S_1, S_2)+\widthf(S_2).
\]
Using tightness of $S'$, $S_1$, and $S_2$, we obtain
\begin{align*}
-\beta
&=
\widthf(S')-\alpha(S')\\
&=
\widthf(S_1)+\widthf(S_2)+\gap(S_1, S_2)
-\alpha(S_1)-\alpha(S_2)-\alpha_j\\
&=
-2\beta+\gap(S_1, S_2)-\alpha_j.
\end{align*}
Therefore, $\alpha_j=\gap(S_1, S_2)-\beta
=
\gap(S_1,S_2)-\beta.$

\paragraph{The ``moreover'' part.} Evaluate the valid facet inequality on
$S_1\cup\{j\}$. Since $S_1\prec\{j\}$,
\begin{align*}
\widthf(S_1\cup\{j\})
-\alpha(S_1\cup\{j\})
&=
\widthf(S_1)+p^j-u(S_1)
-\alpha(S_1)-\alpha_j\\
&=
p^j-u(S_1)-\beta-\alpha_j\\
&=
p^j-\ell(S_2)\\&\ge-\beta.
\end{align*}
where the last inequality is by validity at $S_1 \cup \{j\}$.
Applying the same argument to $S_2\cup\{j\}$ gives $p^j-u(S_1)\le\beta.$
If $\ell(S_2)-p^j=\beta$, the displayed equality also shows that
$S_1\cup\{j\}$ is tight. This set is compatible with every member of
$\mathcal L$ and lies strictly between $S_1$ and $S'$, contradicting
maximality of $\mathcal L$. Hence $\ell(S_2)-p^j<\beta$.
The argument for $p^j-u(S_1)<\beta$ is symmetric, using
$S_2\cup\{j\}$. This completes the proof.
\end{proof}

\manyChildreUnion*

\begin{proof}
By Lemma~\ref{lem:disjoint-tight-gap}, every adjacent gap satisfies $ \operatorname{gap}(S_i,S_{i+1})>\beta.$
Therefore
\begin{align*}
 &\widthf\left(\bigcup_{i=1}^kS_i\right)
 -\alpha\left(\bigcup_{i=1}^kS_i\right)\\
 &\qquad=
 \sum_{i=1}^k[\widthf(S_i)-\alpha(S_i)]
 +\sum_{i=1}^{k-1}\operatorname{gap}(S_i,S_{i+1})\\
 &\qquad>
 -k\beta+(k-1)\beta
 =-\beta.
\end{align*}
Thus the union is not tight.
\end{proof}

\oneChild*
\begin{proof}
Suppose $p^j>u(S)$. Tightness of $S'$ and $S$ gives $\widthf(S')-\bar\alpha(S')=-\bar\beta$ and $\widthf(S)-\bar\alpha(S)=-\bar\beta.$
Subtracting the second equality from the first, and using $S'=S\cup\{j\}$,
gives $\bar\alpha_j=\widthf(S')-\widthf(S)=p^j-u(S)$,
where the last equality is because $\ell(S')=\ell(S)$.

The argument for the case $p^j<\ell(S)$ is symmetric.

Now suppose inserting $j$ into $S$ does not change its width. Then $\widthf(S')=\widthf(S\cup\{j\})=\widthf(S)$. Tightness at $S'$ and $S$ gives $\alpha_j=\widthf(S')-\widthf(S)=0$.
Finally, $\alpha_j\ge0$ by the formula. If $\alpha_j=\beta$, Lemma~\ref{lem:singleton-alpha} would imply that $\{j\}$ is a leaf, contrary to $j$ being introduced at the non-leaf node $S'$. Hence $\alpha_j<\beta$.
\end{proof}

\onedSeparation*

\begin{proof}
First check the trivial inequalities $w\ge0$ and $0\le z_i\le1$.
Otherwise, solve \eqref{eq:mcf}. By Lemma~\ref{lem:dual-mcf} and strong
duality, its optimal value equals that of \eqref{eq:sep-primal}.

The lower bounds in \eqref{eq:mcf} can be removed by the standard flow
shift. Let $\pi_v$ denote optimal node potentials for the resulting
minimum-cost flow problem. Node potentials are the dual variables of
the flow-balance equations, and satisfy $\pi_u-\pi_v\le c_{uv}$ for every arc $(u,v).$
Define $\alpha_i:=\pi_{i^+}-\pi_{i^-},$ and $\beta:=\pi_t-\pi_s.$
The zero-cost point arcs imply $\alpha_i\ge0$, while the zero-cost path
$s\to i^-\to i^+\to t$ implies $\beta\ge0$. Furthermore, for every
interval $[l,r]$,
\[
\begin{aligned}
\sum_{i=l}^r\alpha_i-\beta
&=
(\pi_s-\pi_{l^-})
+\sum_{i=l}^{r-1}
(\pi_{i^+}-\pi_{(i+1)^-})
+(\pi_{r^+}-\pi_t)\\
&\le
\sum_{i=l}^{r-1}(p_{i+1}-p_i)
=
p_r-p_l.
\end{aligned}
\]
Thus $(\alpha,\beta)$ is feasible for \eqref{eq:sep-primal}. The dual
objective of the minimum-cost flow problem is $\sum_{i=1}^n\hat z_i\alpha_i-\beta,$
so optimal node potentials recover an optimal solution of
\eqref{eq:sep-primal}. This use of node potentials and reduced costs is
standard in minimum-cost flow. See
\cite{Orlin1993} as an example. Hence $w-\alpha^\top z\ge-\beta$
is a most violated normalized inequality.

The network has $O(n)$ nodes and arcs. Therefore, using a strongly
polynomial minimum-cost flow algorithm, separation is strongly
polynomial. In particular, Orlin's algorithm gives
$O(n^2\log^2 n)$ time for this network \cite{Orlin1993}.
\end{proof}

\section{Repeated Coordinates}
\label{subsec:repeated-coordinates}

We now remove the assumption that the one-dimensional coordinates are
distinct. Reindex the points, if necessary, so that $p^1\le p^2\le\cdots\le p^n.$
Let $\xi^1<\xi^2<\cdots<\xi^m$ be the distinct coordinate values, and define the corresponding coordinate
classes $C_r:=\{i\in[n]:p^i=\xi^r\},$ where $r\in[m].$
Thus $[n]=C_1\dot\cup C_2\dot\cup\cdots\dot\cup C_m.$
The width of a set depends only on which coordinate classes it intersects.
This suggests introducing one binary variable for the activation of each
class while retaining the original variables to record the multiplicity
within a class.

Define the reduced boxing set on the $m$ distinct coordinates by $\bar B
:=
\{
(w,y)\in\mathbb R\times\{0,1\}^m:
w\ge \overline{\widthf}(S(y))
\},$
where, for $J\subseteq[m]$,
\[
\overline{\widthf}(J)
:=
\begin{cases}
\max_{r\in J}\xi^r-\min_{r\in J}\xi^r,
    &J\neq\emptyset,\\
0,
    &J=\emptyset.
\end{cases}
\]
The set $\bar B$ is precisely the one-dimensional boxing set associated
with the distinct coordinates $\xi^1<\cdots<\xi^m,$
and hence the results of the preceding subsections apply directly to
$\conv(\bar B)$.

\subsection{Class activation and multiplicity}
For each class $C_r$, introduce a class-activation variable $y_r$.
Consider the polyhedron
\begin{equation}\label{eq:repeated-coordinate-extension}
\mathcal R
:=
\left\{
(w,z,y)\in \mathbb R^{1+n+m}:
\begin{array}{l}
(w,y)\in\conv(\bar B),\\[1mm]
z_i\le y_r
    \quad i\in C_r,\ r\in[m],\\[1mm]
y_r\le\displaystyle\sum_{i\in C_r}z_i
    \quad r\in[m],\\[2mm]
0\le z_i\le1
    \quad i\in[n]
\end{array}
\right\}.
\end{equation}

For binary $z$, the two linking inequalities force $y_r
=
\max_{i\in C_r}z_i,$
so $y_r$ records whether at least one point from coordinate class $C_r$
has been selected.

\begin{lemma}
\label{lem:coordinate-class-formulation}
The original boxing polyhedron (with repeated coordinates) satisfies $\conv(B)=\proj_{w,z}(\mathcal R).$
\end{lemma}

\begin{proof}
Let $(w,z)\in \conv(B)$ where $z$ is binary, we find $y\in\{0,1\}^m$ such that $(w,z,y)\in \mathcal R$. Define $y_r:=\max_{i\in C_r}z_i,$ for all $r\in[m].$ Then $y\in\{0,1\}^m$, the linking inequalities in
\eqref{eq:repeated-coordinate-extension} hold, and $\widthf(S(z))
=
\overline{\widthf}(S(y)).$
Thus every point of $B$ has a lift to $\mathcal R$. Since
$\mathcal R$ is convex, we have $\conv(B)\subseteq\proj_{w,z}(\mathcal R).$

For the reverse inclusion, take $(w,z,y)\in\mathcal R.$
Since $(w,y)\in\conv(\bar B)$, it can be expressed as a convex
combination of points whose class-activation vectors are binary, together
with a possible nonnegative multiple of the recession direction.
Thus it suffices to refine such a convex combination so that the
individual points inside every active coordinate class reproduce the
marginals $z$.

The polytope $\{
x\in[0,1]^{C_r}:
\sum_{i\in C_r}x_i\ge1
\}$
is the convex hull of the nonzero binary vectors on $C_r$. Hence there
exists a probability distribution on the nonempty subsets
$T\subseteq C_r$ whose incidence-vector marginals are $x$.

For each binary vector $\bar y$ appearing in the convex combination of
$(w,y)$, we construct a binary vector $\bar z$ on the original index set
as follows. For every coordinate class $C_r$:

\begin{itemize}
    \item if $\bar y_r=0$, set $\bar z_i=0 \ \forall i\in C_r;$

    \item if $\bar y_r=1$, choose a nonempty subset $T_r\subseteq C_r$
    according to the distribution described above, and set
    \[
    \bar z_i=
    \begin{cases}
    1, & i\in T_r,\\
    0, & i\in C_r\setminus T_r.
    \end{cases}
    \]
\end{itemize}

Thus, for every class $C_r$, $\bar y_r=1$ if and only if $\sum_{i\in C_r}\bar z_i\ge1.$
Hence $\bar y$ and $\bar z$ activate exactly the same coordinate classes.

Since all points in $C_r$ have the same coordinate value $\xi_r$, the
width depends only on which coordinate classes are active, and not on how
many points are selected within an active class. Therefore $\widthf(S(\bar z))= \overline{\widthf}(S(\bar y)).$
Consequently $(w,z)$ is a convex combination
of points of $B$, together with a possible nonnegative recession
direction. Hence $(w,z)\in\conv(B).$
\end{proof}

\subsection{Lifting facets from the distinct-coordinate instance}

Let
\begin{equation}\label{eq:reduced-coordinate-facet}
w-\sum_{r=1}^m\bar\alpha_r y_r\ge-\beta
\end{equation}
be a nontrivial facet of $\conv(\bar B)$. By Proposition~4.2, $\bar\alpha_r\ge0$ for all $r\in[m].$

For each $r\in[m]$, choose one representative $\rho_r\in C_r.$
The corresponding \emph{representative lift} of
\eqref{eq:reduced-coordinate-facet} is
\begin{equation}\label{eq:representative-lift}
w-\sum_{r=1}^m\bar\alpha_r z_{\rho_r}\ge-\beta.
\end{equation}

\begin{lemma}
\label{lem:representative-lifting}
Every representative lift
\eqref{eq:representative-lift} of a nontrivial facet
\eqref{eq:reduced-coordinate-facet} is facet-defining for
$\conv(B)$.
\end{lemma}

\begin{proof}
We first prove validity. Let $S\subseteq[n]$, and let $J(S):=\{r\in[m]:S\cap C_r\neq\varnothing\}$
be the set of active coordinate classes. Since the width depends only on
the active classes, $\widthf(S)=\overline{\widthf}(J(S)).$
Moreover, $\sum_{r=1}^m
\bar\alpha_r\chi^S_{\rho_r}
\le
\sum_{r\in J(S)}\bar\alpha_r,$
because $\bar\alpha\ge0$. Validity of
\eqref{eq:reduced-coordinate-facet} therefore gives
\[
\widthf(S)
-
\sum_{r=1}^m\bar\alpha_r\chi^S_{\rho_r}
\ge
\overline{\widthf}(J(S))
-
\sum_{r\in J(S)}\bar\alpha_r
\ge-\beta.
\]
Hence \eqref{eq:representative-lift} is valid.

It remains to prove facetness. Fix any pair of a nontrivial facet~\eqref{eq:reduced-coordinate-facet} of $\conv(\bar B)$ and its representative lift~\eqref{eq:representative-lift}.
Since
\eqref{eq:reduced-coordinate-facet} is a facet of the full-dimensional
polyhedron $\conv(\bar B)$, choose $m+1$ affinely independent tight
points $(\overline{\widthf}(J^s),\chi^{J^s}),$ $J^s\in 2^{[m]}$ indexed by $s=0,\ldots,m.$
For every $s$, define the representative support $R^s:=\{\rho_r\in C_r:r\in J^s\}.$
Then $\widthf(R^s)=\overline{\widthf}(J^s),$
and therefore $(\widthf(R^s),\chi^{R^s}),$
are tight for \eqref{eq:representative-lift} for every $s=0,\ldots,m$. Their projection onto the
coordinates $(w,z_{\rho_1},\ldots,z_{\rho_m})$
is exactly the original collection of $m+1$ affinely independent tight
points, so these lifted points are affinely independent.

Now fix $i\in C_r\setminus\{\rho_r\}.$
Every class $r$ must be contained in some tight support $J^s\subseteq 2^{[m]}$ of a facet of $\conv(\bar B)$. For such an $s$, the support $R^s\cup\{i\}$
has the same width as $R^s$, because both $i$ and $\rho_r$ have coordinate
$\xi^r$. The coefficient of $z_i$ in
\eqref{eq:representative-lift} is zero. Therefore, $(\widthf(R^s),\chi^{R^s\cup\{i\}})$
is also tight. Its difference from
$(\widthf(R^s),\chi^{R^s})$ is exactly the unit vector in coordinate $z_i$.

Constructing $(\widthf(R^s),\chi^{R^s\cup\{i\}})$ for every  $i\in C_r\setminus\{\rho_r\}$ gives $n-m$ additional independent directions. Together with the $m+1$ representative tight points, we obtain $n+1$ affinely independent tight points. Therefore \eqref{eq:representative-lift} defines a facet of $\conv(B)$.
\end{proof}

\subsection{Complete facet description}

We can now extend the distinct-coordinate description to arbitrary
one-dimensional instances.

\begin{theorem}
\label{thm:repeated-coordinate-facets}
Let $C_1,\ldots,C_m$
be the coordinate classes associated with the distinct values $\xi^1<\cdots<\xi^m.$
Consider any nontrivial facet~\eqref{eq:reduced-coordinate-facet} of $\conv(\bar B)$. For every $r$ with $\bar\alpha_r>0$, choose one
representative $\rho_r\in C_r$ for each class $C_r$. The representative
lifts~\eqref{eq:representative-lift} for every $\rho_r\in C_r$ and for all $C_r$ is a nontrivial facet of $\conv(B)$.

Consequently, $\conv(B)$ is described by $w\ge0,$ $0\le z_i\le1,$ for all $i\in[n],$ together with all such representative lifts.
\end{theorem}

\begin{proof}
Lemma~\ref{lem:representative-lifting} shows that every displayed
representative lift is facet-defining and valid for $\conv(B)$.

It remains to show that these inequalities describe the entire polyhedron.
Suppose $(w,z)$ satisfies the coordinate bounds, $w\ge0$, and every
representative lift of every nontrivial facet of $\conv(\bar B)$. For every $r\in[m]$, define $y_r:=\max_{i\in C_r}z_i.$
It follows that $z_i\le y_r$ for all $i\in C_r,$
and $y_r\le\sum_{i\in C_r}z_i.$

Consider an arbitrary nontrivial facet for the reduced polyhedron $\conv(\bar B)$: $w-\sum_{r=1}^m\bar\alpha_r y_r\ge-\beta.$
For each $r$, choose $\rho_r\in
\arg\max_{i\in C_r}z_i.$
By assumption, $(w,z)$ satisfies every representative lift of this
reduced facet, and in particular it satisfies the lift corresponding to
this choice of representatives. So we have
\[
w-\sum_{r=1}^m\bar\alpha_r z_{\rho_r}
=
w-\sum_{r=1}^m\bar\alpha_r y_r
\ge-\beta.
\]
Thus $(w,y)$ satisfies every facet of $\conv(\bar B)$, including its
trivial facets, and therefore $(w,y)\in\conv(\bar B).$
Lemma~\ref{lem:coordinate-class-formulation} now gives $(w,z)\in\conv(B).$

Hence the displayed inequalities form an exact description of
$\conv(B)$. Since every nontrivial inequality in this description is
facet-defining by Lemma~\ref{lem:representative-lifting}, every
nontrivial facet of $\conv(B)$ is one of these representative lifts.
\end{proof}

An immediate structural consequence is worth recording.

\begin{corollary}\label{cor:one-positive-coefficient-per-class}
Let $w-\alpha^\top z\ge-\beta$
be a nontrivial facet of $\conv(B)$. For every coordinate class $C_r$,
at most one variable indexed by $C_r$ has a positive coefficient.
More precisely, there exists a facet coefficient
$\bar\alpha_r$ of the reduced distinct-coordinate instance such that
either $\alpha_i=0$ for all $i\in C_r,$
or there is a unique representative $\rho_r\in C_r$ satisfying $\alpha_{\rho_r}=\bar\alpha_r>0,$ and $\alpha_i=0$ for all $i\in C_r\setminus\{\rho_r\}$.
\end{corollary}

This also gives a direct lifting interpretation of duplicate points.

\begin{remark}[Zero-distance duplicates]
\label{rem:zero-distance-duplicates}
Suppose $\rho_r\in C_r$ has already been assigned the positive coefficient
$\bar\alpha_r$ of a reduced facet, and let
$i\in C_r\setminus\{\rho_r\}$. The exact sequential lifting coefficient
of $z_i$ is zero.

Indeed, since the reduced facet is facet-defining, some tight support
contains class $r$. Using $\rho_r$ as the selected representative of that
class gives a tight support $S$ containing $\rho_r$. Adding $i$ does not
change the width: $\widthf(S\cup\{i\})=\widthf(S).$
Since $i$ and $\rho_r$ occupy the same coordinate, the lifting problem
therefore attains value zero at $S$. Nonnegativity of lifting
coefficients gives the reverse inequality. Thus the duplicate receives
coefficient zero.

Hence zero-distance duplicates do not modify the one-dimensional endpoint
or $\beta$-block structure: the first chosen representative of a
coordinate class carries the class coefficient, while all additional
copies can be lifted with coefficient zero.
\end{remark}

Theorem~\ref{thm:repeated-coordinate-facets} gives a direct extension of
Algorithm~\ref{alg:onedim-facet-generation}: enumerating all reduced facets and all choices of representatives generates every nontrivial facet of the original repeated-coordinate instance.

\subsection{Separation with repeated coordinates}
The separation algorithm for repeated coordinates follows directly from the distinct-coordinate case. Given a point $(\widehat w,\widehat z)$, define $\widehat y_r:=\max_{i\in C_r}\widehat z_i$ for each coordinate class $C_r$. We then apply the minimum-cost flow separation algorithm of Theorem~\ref{thm:separation-1d} to the reduced instance with distinct coordinates $\xi_1<\cdots<\xi_m$ and candidate point $(\widehat w,\widehat y)$. If the algorithm returns a violated inequality, we recover an inequality in the original variables by choosing $\rho_r\in\arg\max_{i\in C_r}\widehat z_i$ and assigning the coefficient of each reduced variable $y_r$ to $z_{\rho_r}$, with zero coefficients for all other variables in $C_r$. By Theorem~\ref{thm:repeated-coordinate-facets}, the resulting inequality is valid for $\conv(B)$ and has exactly the same violation at $(\widehat w,\widehat z)$ as the reduced inequality at $(\widehat w,\widehat y)$. Consequently, separation with repeated coordinates reduces to the distinct-coordinate minimum-cost flow problem on the $m$ coordinate classes.

\section{Computing Facet-admissible Right-hand-sides}
\label{subsec:computing-beta}

Recall that the box-width function
$\operatorname{wd}:2^{[n]}\to\mathbb R_+$
is normalized, monotone, and intersecting submodular.
In particular, $\operatorname{wd}(\emptyset)=
\operatorname{wd}(\{i\})=0$ for every $i\in[n]$.

We first show that $\beta_S^*$ and a maximizing partition
can be computed in polynomial time. We then establish
that every maximizing partition is facet-admissible. Thus, this concludes that it is necessary and sufficient to enumerate all subset of $[n]$

\subsection{A parametric partition problem}\label{subsubsec:parametrization}

For any $\lambda\in\mathbb R$, consider
\begin{equation}\label{eq:H-lambda}
H_S(\lambda)
=
\min_{\mathcal P\in\Pi(S)}
\left\{
\sum_{A\in\mathcal P}(\operatorname{wd}(A)
+\lambda)
\right\}.
\end{equation}
The one-block partition $\{S\}$ has objective value
$\operatorname{wd}(S)+\lambda$. Suppose the optimal value satisfies $H_S(\lambda)=\operatorname{wd}(S)+\lambda$.
Then for any partition $\mathcal P$ with at least two blocks,
\begin{align*}
&\sum_{A\in\mathcal P}\operatorname{wd}(A)
+|\mathcal P|\lambda
\ge \operatorname{wd}(S)+\lambda\\
&\qquad\Longleftrightarrow\quad
\lambda\ge
\frac{\operatorname{wd}(S)
-\sum_{A\in\mathcal P}\operatorname{wd}(A)}
{|\mathcal P|-1}
=\beta_S(\mathcal P).
\end{align*}
Consequently, $
\lambda\ge\beta_S^*.$
Thus $\beta_S^*$ is the smallest parameter $\lambda$
for which the one-block partition is optimal in
\eqref{eq:H-lambda}.

We next establish a primal-dual characterization
of $H_S(\lambda)$.

\begin{lemma}[Primal-dual characterization]
\label{lem:partition-primal-dual}
For every $\lambda\in\mathbb R$,
$H_S(\lambda)=F_S(\lambda)$, where
\begin{equation}\label{eq:partition-primal}
\begin{aligned}
F_S(\lambda)=
\max_{x\in\mathbb R^S}\quad
&x(S)\\
\text{\rm s.t.}\quad
&x(T)\le\operatorname{wd}(T)+\lambda,
&&\emptyset\ne T\subseteq S.
\end{aligned}
\tag{P$_\lambda$}
\end{equation}
\end{lemma}

\begin{proof}
The dual of \eqref{eq:partition-primal} is
\begin{equation}\label{eq:partition-dual}
\begin{aligned}
\min_{y\ge0}\quad&
\sum_{\emptyset\ne T\subseteq S}
(\operatorname{wd}(T)+\lambda)y_T\\
\text{\rm s.t.}\quad&
\sum_{T\ni i}y_T=1,
&&i\in S.
\end{aligned}
\tag{D$_\lambda$}
\end{equation}
The dual constraints are equalities because
the primal variables $x_i$ are unrestricted.
Both programs are feasible and bounded in
objective value. Hence strong LP duality applies.

Every partition $\mathcal P\in\Pi(S)$ induces an
integral feasible solution of
\eqref{eq:partition-dual} by setting $y_T=1$
for $T\in\mathcal P$ and $y_T=0$ otherwise.
Conversely, every integral feasible solution of
\eqref{eq:partition-dual} corresponds to a
partition of $S$. Therefore, it suffices to
show that \eqref{eq:partition-dual} admits an
integral optimal solution.

Among all optimal solutions of
\eqref{eq:partition-dual}, choose $y^*$ maximizing
$\sum_{\emptyset\ne T\subseteq S}|T|^2y_T$.
Such a solution exists since the feasible region
of \eqref{eq:partition-dual} is compact.

We claim that the support
$\mathcal L=\{T\subseteq S:y_T^*>0\}$ is laminar.
Suppose otherwise. Then there exist
$A,B\in\mathcal L$ such that
$A\cap B$, $A\setminus B$, and $B\setminus A$
are all nonempty.

Let $\delta=\min\{y_A^*,y_B^*\}>0$.
Decrease $y_A^*$ and $y_B^*$ by $\delta$,
and increase $y_{A\cap B}^*$ and
$y_{A\cup B}^*$ by $\delta$.
The resulting solution remains feasible because
$\chi^A+\chi^B=\chi^{A\cap B}+\chi^{A\cup B}$. Moreover, intersecting submodularity implies
\begin{align*}
\operatorname{wd}(A)+\operatorname{wd}(B)+2\lambda\ge
\operatorname{wd}(A\cap B)
+\operatorname{wd}(A\cup B)+2\lambda.
\end{align*}
Thus the objective value does not increase.
Since $y^*$ is optimal, the resulting solution
is also optimal.

However, the modification increases
$\sum_T|T|^2y_T$ by
\begin{align*}
&\delta\bigl(
|A\cap B|^2+|A\cup B|^2
-|A|^2-|B|^2
\bigr)=
2\delta|A\setminus B|\,|B\setminus A|>0,
\end{align*}
contradicting the choice of $y^*$.
Hence $\mathcal L$ is laminar.

Restrict \eqref{eq:partition-dual} to the
variables indexed by $\mathcal L$.
Its constraint matrix is the incidence matrix
$[\chi^T]_{T\in\mathcal L}$, which is totally
unimodular because $\mathcal L$ is laminar.
Since the right-hand side is integral, the
restricted LP admits an integral optimal solution.

Extending this solution by zero outside
$\mathcal L$ gives an integral optimal solution
of \eqref{eq:partition-dual}. Its positive
support forms a partition of $S$.
Consequently, the optimal value of
\eqref{eq:partition-dual} equals
$H_S(\lambda)$.
By strong LP duality,
$F_S(\lambda)=H_S(\lambda)$.
\end{proof}

\subsection{Existence of facet-admissible partitions}

We now relate $\beta_S^*$ to the facets
of $\conv(B)$.

Recall that a partition
$\mathcal P=\{S_1,\ldots,S_k\}$
of $S$, with $k\ge2$, is
facet-admissible if there exists a
facet-defining inequality
$w-\alpha^\top z\ge-\beta$
of $\conv(B)$ for which $S$ and
every $S_i$ are tight.

\allSubsetAdmissible*

\begin{proof}
Fix $S\subseteq[n]$ with $|S|\geq 2$.
Since the singleton partition is feasible,
$\beta_S^*\geq \widthf(S)/(|S|-1)\geq 0$.

By the definition of $\beta_S^*$,
$H_S(\beta_S^*)=\widthf(S)+\beta_S^*$.
Lemma~\ref{lem:partition-primal-dual} therefore implies that evaluating 
$(P_\lambda)$ at
$\lambda=\beta_S^*$ has an optimal
solution $x\in\R^S$ satisfying
\begin{align*}
x(T)&\leq\widthf(T)+\beta_S^*,
\quad\emptyset\neq T\subseteq S,\\
x(S)&=\widthf(S)+\beta_S^*.
\end{align*}
Extend $x$ to $\alpha\in\R^n$ by setting
$\alpha_i=x_i$ for $i\in S$ and
$\alpha_i=0$ otherwise. It suffices to show that
$w-\alpha^\top z\geq-\beta_S^*$
is valid for $\conv(B)$.

Take any $T\subseteq[n]$
with $T\cap S\neq\varnothing$,
monotonicity gives
\[
\alpha(T)=x(T\cap S)
\leq\widthf(T\cap S)+\beta_S^*
\leq\widthf(T)+\beta_S^*.
\]
If $T\cap S=\varnothing$, then
$\alpha(T)=0\leq\widthf(T)+\beta_S^*$.
Thus the inequality is valid.
Moreover, $S$ is tight.

Now let $\mathcal P=\{S_1,\ldots,S_k\}$
be any partition attaining $\beta_S^*$.
By the definition of $\beta_S^*$,
\[
\widthf(S)+\beta_S^*
=\sum_{i=1}^k(\widthf(S_i)+\beta_S^*).
\]
Since the $S_i$ partition $S$,
$\sum_{i=1}^k x(S_i)=x(S)$.
Together with the tightness of $S$,
this gives
\[
0=\sum_{i=1}^k
\bigl(\widthf(S_i)+\beta_S^*-x(S_i)\bigr).
\]
Every summand is nonnegative by
feasibility of $x$. Hence every
summand must vanish, and therefore
\[
x(S_i)-\beta_S^*=\widthf(S_i),
\qquad i=1,\ldots,k.
\]
Thus $S,S_1,\ldots,S_k$ are
simultaneously tight for the valid
inequality
$w-\alpha^\top z\geq-\beta_S^*$.
By Lemma~\ref{lem:admissible-partition}, there exists a
facet-defining inequality with the same
right-hand side for which these sets
remain tight. Hence $\mathcal P$ is
facet-admissible.
\end{proof}

\subsection{A polynomial-time algorithm for
$\boldsymbol{\beta_S^*}$}

By the parametrization in Section~\ref{subsubsec:parametrization}, we need to search through $\lambda$ and find the smallest $\lambda$ such that $H_S(\lambda)=\operatorname{wd}(S)+\lambda$. We give a constructive
algorithm that computes $H_S(\lambda)$ and
recovers an optimal partition. This algorithm will be used as a subroutine for computing $\beta_S^*$.

\begin{lemma}
\label{lem:constructive-partition}
For any fixed $\lambda\in\mathbb R$,
$H_S(\lambda)$ and an optimal partition
can be computed using at most $|S|$
submodular-function minimizations.
\end{lemma}

\begin{proof}
Write $S=\{s_1,\ldots,s_m\}$ in an arbitrary
order, and let
$E_j=\{s_1,\ldots,s_j\}$, with $E_0=\emptyset$.

We construct a feasible solution of
\eqref{eq:partition-primal} sequentially.
At iteration $j$, having determined
$x_{s_1},\ldots,x_{s_{j-1}}$, set
\begin{equation}\label{eq:greedy-coordinate}
x_{s_j}
=
\min_{A\subseteq E_{j-1}}
\left\{
\operatorname{wd}(A\cup\{s_j\})
+\lambda-x(A)
\right\}.
\end{equation}

This is a standard submodular-function
minimization problem. Indeed, the function 
$A\mapsto\operatorname{wd}(A\cup\{s_j\})$
is submodular on $E_{j-1}$ because
every set involved in the submodularity
inequality contains $s_j$.
Subtracting the modular function $x(A)$
preserves submodularity.

We first verify that the resulting vector
$x$ is feasible for \eqref{eq:partition-primal}.
For any nonempty $T\subseteq S$, let $s_j$
be the last element of $T$ in the chosen
order. Then $T=A\cup\{s_j\}$ for some
$A\subseteq E_{j-1}$. By
\eqref{eq:greedy-coordinate},
\begin{align*}
x(T)=x(A)+x_{s_j}&\le x(A)+
\operatorname{wd}(A\cup\{s_j\})
+\lambda-x(A)=\operatorname{wd}(T)+\lambda.
\end{align*}

Next, let $A_j$ be any minimizer of
\eqref{eq:greedy-coordinate}, and define
$T_j=A_j\cup\{s_j\}$.
By construction,
\begin{align*}
x(T_j)=x(A_j)+x_{s_j}=\operatorname{wd}(T_j)+\lambda.
\end{align*}
Thus $T_j$ is tight for
\eqref{eq:partition-primal}.
Moreover, $s_j\in T_j$ for every $j$,
so $\bigcup_{j=1}^mT_j=S$.

We claim that the union of any two
intersecting tight sets is also tight.
Indeed, suppose $A$ and $B$ are tight
with $A\cap B\ne\emptyset$.
By intersecting submodularity and
feasibility of $x$,
\begin{align*}
x(A)+x(B)
&=\operatorname{wd}(A)+\operatorname{wd}(B)
+2\lambda\\
&\ge\operatorname{wd}(A\cup B)
+\operatorname{wd}(A\cap B)+2\lambda\\
&\ge x(A\cup B)+x(A\cap B)\\
&=x(A)+x(B).
\end{align*}
Hence all inequalities hold at equality,
and in particular
$x(A\cup B)=\operatorname{wd}(A\cup B)+\lambda$.

Starting with $T_1,\ldots,T_m$,
repeatedly replace any two intersecting
sets by their union. Every resulting set
remains tight, and the union of the
family remains $S$. Since each replacement
decreases the number of sets, the procedure
terminates with a partition
$\mathcal P_\lambda=\{B_1,\ldots,B_k\}$
of $S$ consisting entirely of tight sets.

Consequently,
\begin{align*}
x(S) =\sum_{i=1}^k x(B_i)=\sum_{i=1}^k
\bigl(\operatorname{wd}(B_i)+\lambda\bigr)\ge H_S(\lambda).
\end{align*}
On the other hand, for any partition
$\mathcal Q\in\Pi(S)$, feasibility of $x$
implies
\begin{align*}
x(S) =\sum_{A\in\mathcal Q}x(A)\le\sum_{A\in\mathcal Q}
\bigl(\operatorname{wd}(A)+\lambda\bigr).
\end{align*}
Taking the minimum over $\mathcal Q$ gives
$x(S)\le H_S(\lambda)$.
Therefore,
\begin{equation*}
H_S(\lambda)
=
x(S)
=
\sum_{A\in\mathcal P_\lambda}
\bigl(\operatorname{wd}(A)+\lambda\bigr).
\end{equation*}

Equivalently, assigning dual variable
$y_A=1$ for each $A\in\mathcal P_\lambda$
and zero otherwise gives a feasible
solution of \eqref{eq:partition-dual}
whose objective equals that of $x$.
Thus the construction also provides
a primal-dual optimality certificate.

There are $m$ minimizations in
\eqref{eq:greedy-coordinate}, and merging
the resulting tight sets takes polynomial
time. This proves the result.
\end{proof}

 The idea behind the parametrization is that the parameter \(\lambda\) acts as a penalty on the number of blocks in a partition. Thus, we start with the singleton partition and increase \(\lambda\) until the one-block partition becomes optimal. At each iteration, if the current minimizing partition \(\mathcal P\) has a smaller objective value than the one-block partition, we increase \(\lambda\) to \(\beta_S(\mathcal P)\), at which the two partitions have equal objective values. Since a partition with more blocks has a larger slope with respect to \(\lambda\), any partition with at least \(|\mathcal P|\) blocks can no longer improve upon the one-block partition. Consequently, every unsuccessful iteration produces a partition with strictly fewer blocks than the previous one. This guarantees polynomial termination.

Algorithm~\ref{alg:compute-beta} accomplishes
this using the constructive procedure from
Lemma~\ref{lem:constructive-partition}. In the algorithm, the subroutine 
$\operatorname{MinPartition}(S,\lambda)$ returns a partition attaining $H_S(\lambda)$ using Lemma~\ref{lem:constructive-partition}.

\begin{algorithm}[H]
\caption{Computing $\beta_S^*$}
\label{alg:compute-beta}
\begin{algorithmic}[1]
\Require $S\subseteq[n]$, $m=|S|\ge2$
\Ensure $\beta_S^*$ and a maximizing partition
\State $\mathcal R\gets
       \{\{i\}:i\in S\}$
\State $\lambda\gets
       \operatorname{wd}(S)/(m-1)$
\Loop
    \State $\mathcal P\gets
           \operatorname{MinPartition}(S,\lambda)$
    \If{$\displaystyle
          \sum_{A\in\mathcal P}
          \bigl(\operatorname{wd}(A)+\lambda\bigr)
          =\operatorname{wd}(S)+\lambda$}
        \State \Return $(\lambda,\mathcal R)$
    \EndIf
    \State $\lambda\gets
           \displaystyle
           \frac{\operatorname{wd}(S)-
           \sum_{A\in\mathcal P}\operatorname{wd}(A)}
           {|\mathcal P|-1}$
    \State $\mathcal R\gets\mathcal P$
\EndLoop
\end{algorithmic}
\end{algorithm}

Now we are ready to prove Theorem~\ref{thm:beta-polynomial}.
\betaPolynomial*

\begin{proof}
We establish correctness and polynomial
termination using the following loop
invariants. At the beginning of every
iteration, the current parameter $\lambda$
and reference partition $\mathcal R$ satisfy
\begin{align}
\lambda
&=\beta_S(\mathcal R)\le\beta_S^*,
\label{eq:invariant-1}\\
\sum_{A\in\mathcal Q}\operatorname{wd}(A)
+|\mathcal Q|\lambda
&\ge\operatorname{wd}(S)+\lambda
\quad
\forall\mathcal Q\in\Pi(S)
\text{ with }|\mathcal Q|\ge|\mathcal R|.
\label{eq:invariant-2}
\end{align}

\emph{Initialization.}
Initially, $\mathcal R$ is the singleton
partition. Since every singleton has
zero width,
\begin{align*}
\lambda =\frac{\operatorname{wd}(S)}{m-1} =\beta_S(\mathcal R)
\le\beta_S^*.
\end{align*}
Thus \eqref{eq:invariant-1} holds.
The only partition of $S$ with at least
$m$ blocks is $\mathcal R$, and
\begin{align*}
\sum_{A\in\mathcal R}\operatorname{wd}(A)
+m\lambda =m\lambda=\operatorname{wd}(S)+\lambda.
\end{align*}
Hence \eqref{eq:invariant-2} also holds.

\emph{Termination criterion.}
Suppose the subroutine returns a partition
$\mathcal P$ satisfying
\[
\sum_{A\in\mathcal P}
\bigl(\operatorname{wd}(A)+\lambda\bigr)
=H_S(\lambda)
=\operatorname{wd}(S)+\lambda.
\]
By construction,
$\lambda\ge\beta_S^*$.
Together with \eqref{eq:invariant-1},
this implies $\lambda=\beta_S^*$.
Since $\beta_S(\mathcal R)=\lambda$,
the partition $\mathcal R$ attains
$\beta_S^*$, and the algorithm returns
the desired solution.

\emph{Preservation of the invariants.}
Suppose instead that
$H_S(\lambda)<\operatorname{wd}(S)+\lambda$.
Let $\mathcal P$ be the optimal partition achieving $H_S(\lambda)$
returned by the subroutine. 
\begin{align*}
\sum_{A\in\mathcal P}(\operatorname{wd}(A)
+\lambda) =H_S(\lambda)<\operatorname{wd}(S)+\lambda.
\end{align*}
The one-block partition cannot satisfy
this strict inequality, so $|\mathcal P|\ge2$.
Rearranging yields
\begin{align}\label{eq:updated-tie}
 \lambda
<
\frac{\operatorname{wd}(S)
-\sum_{A\in\mathcal P}\operatorname{wd}(A)}
{|\mathcal P|-1} =:\lambda'.
\end{align}
Thus,
$\lambda<\lambda'\le\beta_S^*$ by the definition
of $\beta_S^*$. This establishes~\eqref{eq:invariant-1}.

We next show~\eqref{eq:invariant-2}. Fix a partition $\mathcal Q$ with 
$|\mathcal Q|\ge |\mathcal P|$.
Since $\mathcal P$ is a minimizer of $H_S(\lambda)$,
\begin{align*}
\sum_{A\in\mathcal Q}(\operatorname{wd}(A)
+\lambda)
\ge
\sum_{A\in\mathcal P}(\operatorname{wd}(A)
+\lambda).
\end{align*}
Since $|\mathcal Q|\ge |\mathcal P|$ and $\lambda'>\lambda$,
\begin{align*}
&
\sum_{A\in\mathcal Q}(\operatorname{wd}(A)
+\lambda')
\ge
\sum_{A\in\mathcal P}(\operatorname{wd}(A)
+\lambda') = \widthf(S) + \lambda'
.
\end{align*}
where the last equality follows from the definition of $\lambda'$ in~\eqref{eq:updated-tie}. This
establishes \eqref{eq:invariant-2}.

\emph{Polynomial termination.}
At every unsuccessful iteration,
the returned partition $\mathcal P$
satisfies
\begin{align*}
\sum_{A\in\mathcal P}\operatorname{wd}(A)
+|\mathcal P|\lambda
=H_S(\lambda)<\operatorname{wd}(S)+\lambda.
\end{align*}
By \eqref{eq:invariant-2}, this is
possible only if
$|\mathcal P|<|\mathcal R|$ where $\mathcal R$ is from the previous iteration.
Therefore, the number of blocks in
the reference partition strictly
decreases at every unsuccessful iteration.

Initially, $\mathcal R$ has $m$ blocks,
and every unsuccessful iteration
produces a partition with at least
two blocks. Hence there are at most
$m-2$ unsuccessful iterations,
followed by one successful iteration. Each iteration requires at most $m$
submodular-function minimizations by
Lemma~\ref{lem:constructive-partition}.
Thus the total number is at most
$m(m-1)=O(m^2)$. Since submodular-function minimization
admits strongly polynomial algorithms,
the result follows.
\end{proof}

\section{Computing $\beta_{\min}$}\label{sec:findBetaMin}
We show that $\beta_{\min}$ and a subset attaining
the minimum can be computed in strongly polynomial
time by exploiting the intersecting submodularity
of the box-width function.

For any $\lambda\in\mathbb R$, consider the
parametric optimization problem
\begin{equation}\label{eq:beta-min-parametric}
M(\lambda)
=
\min_{\emptyset\ne S\subseteq[n]}
\left\{
\operatorname{wd}(S)-\lambda(|S|-1)
\right\}.
\end{equation}

Since every singleton has zero width,
$M(\lambda)\le0$ for every $\lambda$.
Moreover, by the definition of $\beta_{\min}$,
\begin{align}
M(\lambda)=0
&\iff
\operatorname{wd}(S)-\lambda(|S|-1)\ge0
\quad\forall S\subseteq[n],\ |S|\ge2
\nonumber\\
&\iff
\lambda\le
\min_{\substack{S\subseteq[n]\\|S|\ge2}}
\frac{\operatorname{wd}(S)}{|S|-1}
\nonumber\\
&\iff
\lambda\le\beta_{\min}.
\label{eq:beta-min-characterization}
\end{align}
Thus $\beta_{\min}$ is the largest parameter
$\lambda$ for which $M(\lambda)=0$.

Although the box-width function is not
submodular on all subsets, the parametric
problem \eqref{eq:beta-min-parametric} can
be reduced to ordinary submodular-function
minimization by fixing an element of the subset.

\begin{lemma}
\label{lem:beta-min-sfm}
For any fixed $\lambda\in\mathbb R$,
$M(\lambda)$ and a subset attaining
$M(\lambda)$ can be computed using at most
$n$ submodular-function minimizations.
\end{lemma}

\begin{proof}
For every $i\in[n]$, consider
\begin{equation}\label{eq:beta-min-anchored}
\min_{A\subseteq[n]\setminus\{i\}}
\left\{
\operatorname{wd}(A\cup\{i\})-\lambda|A|
\right\}.
\end{equation}

We first show that \eqref{eq:beta-min-anchored}
is a submodular-function minimization problem.
For any $A,B\subseteq[n]\setminus\{i\}$,
the sets $A\cup\{i\}$ and $B\cup\{i\}$
have a nonempty intersection. Therefore,
intersecting submodularity of the box-width
function implies
\begin{align*}
\operatorname{wd}(A\cup\{i\})
+\operatorname{wd}(B\cup\{i\})\ge
\operatorname{wd}((A\cap B)\cup\{i\})
+\operatorname{wd}((A\cup B)\cup\{i\}).
\end{align*}
Hence the function
$A\mapsto\operatorname{wd}(A\cup\{i\})$
is submodular on $[n]\setminus\{i\}$.
Since $A\mapsto\lambda|A|$ is modular,
subtracting it preserves submodularity.

Now every nonempty subset $S\subseteq[n]$
contains some element $i$. Writing
$A=S\setminus\{i\}$ gives
\begin{align*}
\operatorname{wd}(S)-\lambda(|S|-1)
&=\operatorname{wd}(A\cup\{i\})-\lambda|A|.
\end{align*}
Consequently,
\begin{equation}\label{eq:beta-min-sfm-reduction}
M(\lambda)
=
\min_{i\in[n]}
\min_{A\subseteq[n]\setminus\{i\}}
\left\{
\operatorname{wd}(A\cup\{i\})-\lambda|A|
\right\}.
\end{equation}

Thus $M(\lambda)$ can be evaluated by
solving \eqref{eq:beta-min-anchored}
once for each $i\in[n]$.
If $A_i$ is a minimizer for index $i$,
then a subset attaining $M(\lambda)$
is obtained by choosing the best set
among $A_i\cup\{i\}$, $i\in[n]$.
This requires at most $n$
submodular-function minimizations.
\end{proof}

By \eqref{eq:beta-min-characterization},
computing $\beta_{\min}$ amounts to finding
the largest $\lambda$ satisfying $M(\lambda)=0$.

We initialize $\lambda$ using a pair of
points with minimum $\ell_1$-distance.
For any two distinct indices $i,j$,
$\operatorname{wd}(\{i,j\})=\|p_i-p_j\|_1$.
Therefore,
\begin{align*}
\beta_{\min}
\le
\min_{\substack{i,j\in[n]\\i<j}}
\operatorname{wd}(\{i,j\})=
\min_{\substack{i,j\in[n]\\i<j}}
\|p_i-p_j\|_1.
\end{align*}

Starting from this upper bound, the algorithm
repeatedly solves \eqref{eq:beta-min-parametric}.
If a subset with negative objective value
is found, its ratio in \eqref{eq:beta-min}
provides a strictly smaller upper bound.

\begin{algorithm}[t]
\caption{Computing $\beta_{\min}$}
\label{alg:compute-beta-min}
\begin{algorithmic}[1]
\Require Points $p_1,\ldots,p_n\in\mathbb R^d$,
         $n\ge2$
\Ensure $\beta_{\min}$ and a minimizing subset
\State Choose
       $R\in\displaystyle
       \arg\min_{\substack{\{i,j\}\subseteq[n]\\i<j}}
       \operatorname{wd}(\{i,j\})$
\State $\lambda\gets\operatorname{wd}(R)$
\Loop
    \For{$i=1,\ldots,n$}
        \State Compute
        $A_i\in\displaystyle
        \arg\min_{A\subseteq[n]\setminus\{i\}}
        \left\{
        \operatorname{wd}(A\cup\{i\})-\lambda|A|
        \right\}$
    \EndFor
    \State Choose
    $i^*\in\displaystyle
    \arg\min_{i\in[n]}
    \left\{
    \operatorname{wd}(A_i\cup\{i\})
    -\lambda|A_i|
    \right\}$
    \State $T\gets A_{i^*}\cup\{i^*\}$
    \If{$\operatorname{wd}(T)-\lambda(|T|-1)=0$}
        \State \Return $(\lambda,R)$
    \EndIf
    \State $\lambda\gets
    \displaystyle\frac{\operatorname{wd}(T)}{|T|-1}$
    \State $R\gets T$
\EndLoop
\end{algorithmic}
\end{algorithm}

\polytimeBetaMin
\begin{proof}
We first prove correctness and then establish
polynomial termination.

\emph{Loop invariants.}
At the beginning of every iteration,
the current parameter $\lambda$ and
reference subset $R$ satisfy
\begin{align}
\lambda
&=\frac{\operatorname{wd}(R)}{|R|-1}
\ge\beta_{\min},
\label{eq:beta-min-invariant-1}\\
\lambda
&\le
\min_{\substack{i,j\in[n]\\i<j}}
\operatorname{wd}(\{i,j\}).
\label{eq:beta-min-invariant-2}
\end{align}

Both invariants hold at initialization:
$R$ is a minimum-width pair, and
$\lambda=\operatorname{wd}(R)$.
In particular, $\lambda$ is the ratio
associated with an actual subset of
cardinality at least two, so
$\lambda\ge\beta_{\min}$.

\emph{Termination criterion.}
Suppose the subset $T$ returned by the
submodular minimizations satisfies
\[
\operatorname{wd}(T)-\lambda(|T|-1)=0.
\]
By Lemma~\ref{lem:beta-min-sfm}, $T$
attains $M(\lambda)$, so $M(\lambda)=0$.
Equation~\eqref{eq:beta-min-characterization}
then gives $\lambda\le\beta_{\min}$.
Combining this with
\eqref{eq:beta-min-invariant-1},
we conclude that $\lambda=\beta_{\min}$.
Moreover, the reference subset $R$
attains this value by
\eqref{eq:beta-min-invariant-1}.
Thus the returned solution is correct.

\emph{Preservation of the invariants.}
Suppose the algorithm does not terminate.
Since every singleton has zero objective value,
the optimal value of
\eqref{eq:beta-min-parametric} must be negative:
\begin{equation}\label{eq:beta-min-negative}
\operatorname{wd}(T)-\lambda(|T|-1)
=M(\lambda)<0.
\end{equation}

In particular, $|T|\ge2$, so the update $\lambda'
=
\frac{\operatorname{wd}(T)}{|T|-1}$
is well-defined. Rearranging
\eqref{eq:beta-min-negative} gives
\begin{align*}
\operatorname{wd}(T)
&<\lambda(|T|-1),\\
\lambda'
&<\lambda.
\end{align*}

Furthermore, since $\lambda'$ is the
ratio associated with a subset of
cardinality at least two,
$\lambda'\ge\beta_{\min}$.
Together with $\lambda'<\lambda$,
this shows that both
\eqref{eq:beta-min-invariant-1} and
\eqref{eq:beta-min-invariant-2}
are preserved after updating
$(\lambda,R)\gets(\lambda',T)$.

\emph{Polynomial termination.}
We first observe that every subset $T$
returned at an unsuccessful iteration
has at least three elements.
Indeed, by \eqref{eq:beta-min-invariant-2},
every pair $\{i,j\}$ satisfies
\begin{align*}
\operatorname{wd}(\{i,j\})-\lambda
\ge
\operatorname{wd}(\{i,j\})
-\min_{r<s}\operatorname{wd}(\{r,s\})\ge0.
\end{align*}
A singleton also has objective value zero.
Hence a subset satisfying
\eqref{eq:beta-min-negative} must have
cardinality at least three.

We next show that the cardinalities of
the subsets returned at successive
unsuccessful iterations strictly decrease. Suppose $T$ is a minimizing subset at
the current parameter $\lambda$, with
$k=|T|\ge3$, and let
$\lambda'=\operatorname{wd}(T)/(k-1)$.
By optimality of $T$,
for every nonempty $U\subseteq[n]$,
\[
\operatorname{wd}(U)-\lambda(|U|-1)
\ge
\operatorname{wd}(T)-\lambda(k-1).
\]
Therefore, for any $U$ with $|U|\ge k$,
\begin{align*}
&\operatorname{wd}(U)-\lambda'(|U|-1)\\
&=
\operatorname{wd}(U)-\lambda(|U|-1)
+(\lambda-\lambda')(|U|-1)\\
&\ge
\operatorname{wd}(T)-\lambda(k-1)
+(\lambda-\lambda')(|U|-1)\\
&=
(\lambda'-\lambda)(k-1)
+(\lambda-\lambda')(|U|-1)\\
&=
(\lambda-\lambda')(|U|-k)\\
&\ge0.
\end{align*}

Consequently, at the updated parameter
$\lambda'$, every subset with at least
$k$ elements has nonnegative objective
value in \eqref{eq:beta-min-parametric}.

If the next iteration is unsuccessful,
its minimizing subset $T'$ must satisfy $\operatorname{wd}(T')-\lambda'(|T'|-1)<0.$
The preceding argument then implies
$|T'|<k=|T|$.

Thus the cardinalities of the subsets
returned at successive unsuccessful
iterations form a strictly decreasing
sequence of integers in $\{3,\ldots,n\}$.
There can be at most $n-2$
unsuccessful iterations, followed by
one successful iteration.

By Lemma~\ref{lem:beta-min-sfm}, each
iteration requires at most $n$
submodular-function minimizations.
Therefore, the total number of such
minimizations is at most $n(n-1)=O(n^2).$
Since submodular-function minimization
admits strongly polynomial algorithms,
the result follows.
\end{proof}